\documentclass[11pt]{amsart}
\usepackage[cp1251]{inputenc}
\usepackage{color}
\usepackage{amsthm,amssymb,amsmath,latexsym}
\usepackage{graphicx}
\usepackage[T2A]{fontenc}
\usepackage{graphs}

\theoremstyle{definition}
\newtheorem{defin}{Definition}[section]
\theoremstyle{remark}
\newtheorem{example}[defin]{Example}

\newtheorem{remar}[defin]{Remark}
\theoremstyle{plain}
\newtheorem{thm}[defin]{Theorem}
\newtheorem{prop}[defin]{Proposition}
\newtheorem{lemm}[defin]{Lemma}
\newtheorem{corol}[defin]{Corollary}

\def\rank{{\rm rank}}

\def\wt{\widetilde}
\def\ol{\overline}
\def\ov{\overline}
\def\tl{\widetilde}
\def\wh{\widehat}

\numberwithin{equation}{section}

\begin{document}

\title{Subdiagrams and invariant measures on Bratteli diagrams}

\date{}

\author{M. Adamska}
\address{The University of Warmia and Mazury, Olsztyn, Poland}
\email{malwina.adamska@gmail.com}

\author{S. Bezuglyi}
\address{Department of Mathematics, Institute for Low Temperature Physics, Kharkiv, Ukraine\\
{\em Current address: Department of Mathematics, University of Iowa, Iowa City,
52242 IA, USA}}
\email{bezuglyi@gmail.com}

\author{O. Karpel}
\address{Department of Mathematics, Institute for Low Temperature Physics, Kharkiv, Ukraine}
\email{helen.karpel@gmail.com}

\author{J. Kwiatkowski}
\address{Kotarbinski  University of Information Technology and Management, Olsztyn, Poland}
\email{jkwiat@mat.umk.pl}

\dedicatory{Dedicated to the memory of Ola Bratteli.}

\begin{abstract}

We study ergodic finite and infinite measures defined on the path space $X_B$ of a Bratteli diagram $B$ which are invariant with respect to the tail equivalence relation on $X_B$. Our interest is focused on measures supported by vertex and edge subdiagrams of $B$. We give several criteria when a finite invariant measure defined on the path space of a subdiagram of $B$ extends to a finite invariant measure on $B$. Given a finite ergodic measure on a Bratteli diagram $B$ and a subdiagram $B'$ of $B$, we find the necessary and sufficient conditions under which the measure of the path space $X_{B'}$ of $B'$ is positive. For a class of Bratteli diagrams of finite rank, we determine when they have maximal possible number of ergodic invariant measures. The case of diagrams of rank two is completely studied. We include also an example which explicitly illustrates the proved results.

\end{abstract}

\maketitle

\section{Introduction and background}\label{introduction}

\subsection{Motivation and main results} A Bratteli diagram $B$ is an $\mathbb N$-graded graph whose vertices are partitioned into levels and edges connect vertices of consecutive levels. Bratteli diagrams is one of the most studied objects in the theories of operator algebras and dynamical systems.
They were originally defined by O. Bratteli in his breakthrough article on the classification of AF $C^*$-algebras \cite{Br72}. During the last two decades, Bratteli diagrams turned out to be a very powerful and productive tool for the study of dynamical systems in the measurable, Borel, and Cantor settings. Due to the pioneering works by Vershik \cite{vershik:1981}, Herman-Putnam-Skau \cite{herman_putnam_skau:1992}, and Giordano-Putnam-Skau  \cite{giordano_putnam_skau:1995}, one can state, with some abuse of rigor, that every transformation in all mentioned dynamics admits a realization as a continuous map (usually called a Vershik map) acting on the path space of a Bratteli diagram. These remarkable results have served for a long time as the basis for the further study of diverse aspects of Bratteli diagrams and dynamical systems defined on their path spaces. The list of related papers is very long, so that we would refer only to a recent survey by Durand \cite{durand:2010} where the reader will find a detailed exposition of this subject and further references.

It is difficult to overestimate the significance of Bratteli diagrams  for the study  of dynamical systems. The main reason of their importance is the fact that  various properties of dynamical systems  become more transparent and can be investigated  in an easier way when one deals with corresponding Bratteli diagrams and Vershik maps. Our interest is focused on a Cantor dynamical system $(X, T)$ which is determined by a Cantor set $X$ and an aperiodic self-homeomorphism $T$. The problem of finding all ergodic invariant measures and their supports for a given $(X, T)$ is traditionally a central one in the theory of dynamical systems, especially for specific interesting examples of  homeomorphisms $T$.
But being considered in general settings, this problem looks rather vague, and there are very few universal results that can be applied to the study of a given homeomorphism $T$.

The situation changes completely when one first realizes $T$ as a Vershik map acting on the path space $X_B$ of a Bratteli diagram $B$. Then the structure of the corresponding Bratteli diagram helps very much to clarify dynamical properties of $T$. In particular, we can say a lot about minimal components,  ergodic invariant measures and their supports. Moreover, if a Bratteli diagram does not admit a Vershik map (this situation was considered  in  \cite{medynets:2006}, \cite{BKY14}, \cite{BY}), we can still  study invariant measures and minimal components with respect to the tail equivalence relation defined on the path space of the diagram.

The present paper continues the series of our articles \cite{BKMS_2010}, \cite{S.B.O.K.}, \cite{BKMS_2013}, and \cite{BKK_14} which are devoted to the study of ergodic measures on Bratteli diagrams from different points of view. In \cite{BKMS_2010}, \cite{S.B.O.K.}, and \cite{BKMS_2013} we dealt with important but peculiar  cases of stationary and finite rank Bratteli diagrams. In this paper, we focus on the study of measures defined on general Bratteli diagrams and which are invariant with respect to the tail equivalence relation.

Suppose that we have a Bratteli diagram $B$ and an ergodic measure $\mu$. It is still an open question whether one can  explicitly describe the support of $\mu$ on $X_B$ in terms of the diagram $B$.
It is worth noting  that  this question is completely solved for stationary and finite rank Bratteli  diagrams, simple or non-simple ones \cite{BKMS_2010}, \cite{BKMS_2013}. In this case such measures with necessity are supported by a subdiagram of $B$ (more precisely, the measure is extended by invariance from a subdiagram). By a Bratteli subdiagram, we mean a Bratteli diagram $B'$ that can be obtained from $B$ by removing some vertices and edges from each level of $B$. Then  $X_{B'}\subset X_B$. We will consider two extreme cases of Bratteli subdiagrams: vertex subdiagram (when we fix a subset of vertices at each level and take all edges between them) and edge subdiagram (some edges are removed from the initial Bratteli diagram but the vertices are not changed). It is clear that an arbitrary subdiagram can be obtained as a combination of these cases.

The central problems, we are interested in this paper, are the following:

(A) Given a subdiagram $B'$ of $B$ and an ergodic measure $\mu$ on $X_B$, under what conditions on $B'$ the subset $X_{B'}$ has positive measure $\mu$ in $X_B$?

(B) Let $\nu$ be a measure supported by the path space $X_{B'}$ of a  subdiagram $B' \subset B$. Then $\nu$ is extended to the subset $\mathcal E(X_{B'})$ by invariance with respect to the tail equivalence relation $\mathcal E$. Under what conditions $\nu(\mathcal E(X_{B'}))$ is finite (or infinite)?

(C) Let $B$ be a  Bratteli diagram of finite rank $k$. It is known that $B$ can support at most $k$ ergodic (finite and infinite) measures. Is it possible to determine under what conditions on the incidence matrices of $B$ there exist exactly $k$ ergodic measures?

Our main results give affirmative answers (in some cases, partial answers) to the questions above.

The paper is organized as follows.  We collected in Introduction all necessary information about Bratteli diagrams and measures on their path spaces. This section also contains basic notation. In Section~\ref{fineteness}, we consider an extension of a probability invariant measure defined on a subdiagram of a Bratteli diagram. We prove  several criteria of the finiteness of the measure extension for both vertex and edge subdiagrams. These results form an extended affirmative answer to question (B) and  moreover refine the statements proved earlier in  \cite{BKMS_2013} and \cite{BKK_14}. In Section~\ref{number_of_erg}, we find out when a finite rank Bratteli diagram has the maximal possible number of finite ergodic invariant measures using the results from Section~\ref{fineteness}. This answers question (C) above. We observe that this type of results is principally new in the context of Bratteli diagrams and is based on the careful study of incidence matrices. Section~\ref{subdiagrams} contains our main results about measure of the path space of a subdiagram. They answer question (A). Given a finite invariant measure $\mu$ on a Bratteli diagram $B$, we determine when  the measure of a path space of a subdiagram $\ov B$ is positive or zero. One of our  results is a criterion for the path space $X_{\ov B}$ to have measure zero. We show that in case when $\mu(X_{\ov B}) = 0$, the extension of any measure defined on $X_{\ov B}$ to the whole space $X_B$ is infinite. In Section~\ref{Example section}, we illustrate the proved results by considering an interesting class of  Bratteli diagrams for which  our results admit an explicit clarification.

\subsection{Invariant measures for Cantor dynamical systems}

Let $(X,T)$ be an aperiodic Cantor dynamical system, that is $X$ is a zero-dimensional compact metric space with no isolated points (a Cantor set) and $T$ is a homeomorphism of $X$ with infinite orbits. A Borel measure $\mu$ is called $T$-\textit{invariant} if $\mu(TA) = \mu(A)$ for any Borel set $A$. Let $M(X,T)$ be the set of all invariant measures. It is known that  $M(X,T)$ is a Choquet simplex whose extreme points are $T$-ergodic measures. This simplex includes probability measures (when $\mu(X) =1$)  and infinite measures (when $\mu(X) = \infty$). We observe that infinite measures cannot arise, for instance, in minimal dynamics.

It follows from \cite{herman_putnam_skau:1992}, \cite{giordano_putnam_skau:1995}, and \cite{medynets:2006} that any minimal (and even aperiodic) Cantor dynamical system $(X,T)$ admits a realization as a Bratteli-Vershik dynamical system $(X_B, \varphi_B)$ acting on a path space $X_B$ of a Bratteli diagram (the definitions of notions related to Bratteli diagrams are given in the next subsection). Thus, the study of measures invariant with respect to a homeomorphism $T$ is reduced to the case of measures defined on the path space of a Bratteli diagram. The advantage of this approach is based on the facts that (i) any such a measure is completely determined by its values on cylinder sets of $X_B$, and (ii) there are simple and explicit formulas for measures of cylinder sets. Especially transparent this method works for stationary and finite rank Bratteli diagrams, simple and non-simple ones \cite{BKMS_2010}, \cite{BKMS_2013}. We remark that any Bratteli diagram of finite rank $k$ has at most $k$ finite and infinite ergodic invariant measures.

We need to point out that the study of measures on a Bratteli diagram is a more general problem than that in the case of Cantor dynamics. This observation follows from the existence of Bratteli diagrams that do not support any continuous dynamics on their path spaces which is compatible with the tail equivalence relation. The first example of such a Bratteli diagram was given  in \cite{medynets:2006}; a more comprehensive coverage of this subject can be found in \cite{BKY14} and \cite{BY}. If a Bratteli diagram does not admit a Bratteli-Vershik homeomorphism, then we consider the {\em tail equivalence relation} $\mathcal E$ on $X_B$ and study measures invariant with respect to $\mathcal E$. In this paper we work with $\mathcal E$-invariant measures.

\subsection{Bratteli diagrams and their subdiagrams}\label{diagram_subdiagram}
The concept of a Bratteli diagram has been studied in a great number of recent research and survey papers devoted to various aspects of Cantor dynamics. We focus here only on some  necessary definitions referring to the pioneering articles \cite{herman_putnam_skau:1992}, \cite{giordano_putnam_skau:1995} (see also  \cite{durand:2010}, \cite{BKMS_2010} and the references there) where the reader can find more facts about Bratteli diagrams and widely used techniques, for instance, the  telescoping procedure. The most of definitions and notation utilized  in this paper are taken from  \cite{BKMS_2013}.

A {\it Bratteli diagram} is an infinite graph $B=(V,E)$ such that the vertex
set $V =\bigcup_{i\geq 0}V_i$ and the edge set $E=\bigcup_{i\geq 1}E_i$
are partitioned into disjoint subsets $V_i$ and $E_i$ where

(i) $V_0=\{v_0\}$ is a single point;

(ii) $V_i$ and $E_i$ are finite sets, $\forall i \geq 1$;

(iii) there exist $r : V \to E$ (range map $r$) and $s : V \to E$ (source map $s$), both from $E$ to
$V$, such that $r(E_i)= V_i$, $s(E_i)= V_{i-1}$, and
$s^{-1}(v)\neq\emptyset$, $r^{-1}(v')\neq\emptyset$ for all $v\in V$
and $v'\in V\setminus V_0$.

Given a Bratteli diagram $B$, the $n$-th {\em incidence matrix}
$F_{n}=(f^{(n)}_{v,w}),\ n\geq 0,$ is a $|V_{n+1}|\times |V_n|$
matrix such that $f^{(n)}_{v,w} = |\{e\in E_{n+1} : r(e) = v, s(e) = w\}|$
for  $v\in V_{n+1}$ and $w\in V_{n}$. Here the symbol $|\cdot |$ denotes the cardinality of a set.

The set of all infinite paths in $B = (V,E)$ is denoted by $X_B$.  The topology defined by finite paths (cylinder sets) turns $X_B$ into a 0-dimensional metric compact space. By assumption, we will consider only such Bratteli diagrams for which $X_B$ is a {\em Cantor set}.

 Let $x =(x_n)$ and $y =(y_n)$  be two paths from $X_B$. It is said that $x$ and $y$ are {\em tail equivalent} (in symbols,  $(x,y) \in \mathcal E$)  if there exists some $n$ such that $x_i = y_i$ for all $i\geq n$. Since $X_B$ has no isolated points, the $\mathcal E$-orbit of any point $x\in X_B$ is countable, i.e., $\mathcal E$ is a Borel countable equivalence relation.

For a Bratteli diagram $B = (V,E)$, we consider subsets  $\ol V\subset V$ and $\ol E \subset E$. Then, by definition, the pair $\ol B = (\ol V,\ol E)$ defines a {\em subdiagram} of $B$ if  $\ol V = s(\ol E)$ and $s(\ol E)  = r(\ol E) \cup \{v_0\}$. The path space $X_{\ol B}$ of the subdiagram $\ol B$ is obviously a closed subset of $X_B$.

On the other hand, there are closed subsets of $X_B$ which are not obtained as the path space of a Bratteli subdiagram. It was proved in \cite{giordano_putnam_skau:2004} that a closed subset $Z \subset X_B$ is the path space of a subdiagram if and only if $\mathcal E|_{Z \times Z}$ is an etal\'{e} equivalence relation (see the definition in \cite{giordano_putnam_skau:2004}).

Let  $\overline{B}$ be a subdiagram of a Bratteli diagram $B$. Then we  consider the sequence of incidence matrices $\{\overline{F}_n\}_{n=0}^{\infty}$ of $\overline{B}$. We will study  two extreme cases of subdiagrams, {\em edge subdiagrams} and {\em vertex subdiagrams}. By definition, an {\em edge} subdiagram is obtained from the diagram $B$ by ``removing'' some edges and leaving all vertices  of $B$ unchanged. That is, the incidence matrices $\{\overline{F}_n\}_{n=0}^{\infty}$ of an edge subdiagram are $|V_{n+1}| \times |V_n|$ matrices such that $\overline{F}_n \leq F_n$ for every $n \in \mathbb{N}$. We denote $\widetilde{F}_n = F_n - \overline{F}_n$. It is also required  that the following property holds as a part of the definition of an edge subdiagram: if $f_{vw}^{(n)} > 0$, then $\overline{f}_{vw}^{(n)} > 0$. Let $\overline{\mathcal{E}}$ denote the tail equivalence relation considered on the set $X_{\overline{B}}$. If there exists $N$ such that $\overline{F}_n = F_n$ for all $n > N$, then the diagram $\overline{B}$ is Kakutani equivalent to $B$; this case was studied earlier (see, for instance, \cite{giordano_putnam_skau:1995}).
So that we will assume, without loss of generality, that  $\overline{F}_n < F_n$ for infinitely many $n$.

To define a {\em vertex} subdiagram, we start with  a sequence $\overline W = \{W_n\}_{n>0}$ of proper subsets $W_n$ of $V_n$, and let  $W'_n = V_n \setminus W_n \neq \emptyset$ for all $n$. The vertex subdiagram $\overline B =  (\ol W, \ol E)$ is formed by the vertices from $W_n$ and by the set of edges $\ol E$ whose source and range are in $W_{n}$ and $W_{n+1}$, respectively. Thus, the incidence  matrix $\ol F_n$ of $\ol B$ has the size $|W_{n+1}| \times |W_n|$, and  it is represented by a block of  $F_n$ corresponding to the vertices from $W_{n}$ and $W_{n+1}$. We say in this case that $\ol W = (W_n)$ is the support of $\ol B$.

It is not hard to see that the study of any subdiagram $\ol B$ of $B$ is reduced to the cases of edge and vertex subdiagrams. Indeed, $\ol B$ can be viewed as an edge subdiagram of a vertex subdiagram of $B$.

Throughout the paper we keep the following notation: given $n \geq 1$ and $v \in V_n$, let $X_v^{(n)}\subset X_B$ be the set of all paths that go through the vertex $v$ (the clopen set $X_v^{(n)}$ is also called a tower because any two cylinder sets (finite paths) $e_1(v_0,v)$ and $e_2(v_0,v)$ from $X_v^{(n)}$ are $\mathcal E$-equivalent); let $h_v^{(n)}$ be the cardinality of  the set $E(v_0, v)$ of all cylinder sets between $v_0$ and $v$, i.e. $h_v^{(n)}$ is the height of the tower   $X_v^{(n)}$. Similarly,  $\overline{X}_v^{(n)}$ stands for the tower in a subdiagram $\overline{B}$ that is determined by  a vertex $v$ of $\ol B$. Thus,  we consider the paths in $\overline{X}_v^{(n)}$ that contain edges from $\overline{B}$ only. Let $\overline{h}_v^{(n)}$ be the  height of the tower $\ol X_v^{(n)}$. As a rule, objects related to a subdiagram $\ol B$ are denoted by barred symbols.

We note that the vectors of tower heights $h^{(n)} = (h^{(n)}_v : v \in V_n)$ are related in the following obvious way:
\begin{equation}\label{formula for heights}
F_n h^{(n)} = h^{(n+1)}, \ \ n \geq 1.
\end{equation}

\subsection{Measures on Bratteli diagrams}\label{measures on diagrams}

Let $\mu$ be a probability $\mathcal E$-invariant Borel measure  on $X_B$ (for brevity, we will use the term ``measure on $B$'' below). It is not hard to see that  $\mu$ is completely determined by its values $(p_v^{(n)}, v \in V_n, n \geq 1)$ on the cylinder sets $[e(v_0,v)]$ corresponding to a finite path between $v_0$ and $v$ (since $\mu$ is $\mathcal E$-invariant, the value $p_v^{(n)}$ does not actually depend on a choice of $e(v_0,v)$).  The details can be found, for instance, in \cite{BKMS_2010}. In other words, $p_v^{(n)} = \dfrac{\mu(X_v^{(n)})}{h_v^{(n)}}$. Then one has the relation
\begin{equation}\label{formula for measures}
 F_n^T p^{(n+1)} = p^{(n)}, \ n\geq 1,
\end{equation}
where $p^{(n)}$ is the column vector with entries $(p^{(n)}_v : v \in V_n)$, and $F_n^T$ denotes the transpose of $F_n$.

Let $\ol B$ be any (vertex or edge) subdiagram  of a Bratteli diagram $B$. Then we can consider an ``inner'' probability measure $\ol\mu$ defined on the path space $X_{\ol B}$ and invariant with respect to $\mathcal {\ol E}$. It is obvious that $\mathcal E|_{X_{\ol B}\times X_{\ol B}} = \mathcal {\ol E}$. For every vertex $v$ from the $n$-th level of $\ol B$, we have $\overline{p}_v^{(n)} = \dfrac{\overline{\mu}(\overline{X}_v^{(n)})}{\overline{h}_v^{(n)}}$,  $n \geq 1$.

Let $\widehat X_{\ol B} := \mathcal E(X_{\ol B})$ be the subset of paths in $X_B$ that are tail equivalent to paths from $X_{\ol B}$. In other words, the $\mathcal E$-invariant subset $\widehat X_{\ol B} $ of $X_B$ is the saturation of $X_{\ol B}$ with respect to the equivalence relation $\mathcal E$ (or $X_{\ol B}$ is a countable complete section of  $\mathcal E$ on $\widehat X_{\ol B}$). Let $\ov \mu$ be an ergodic probability measure on $X_{\ol B}$ invariant with respect to the tail equivalence relation $\mathcal{\ol E}$ defined on $\ol B$. Then $\ov \mu$ can be canonically extended to the ergodic measure $\widehat{\ov \mu}$ on the space $\widehat X_{\ol B}$ by invariance with respect to $\mathcal E$ \cite{BKMS_2013}. In case of need, we can think that $\widehat{\ov \mu}$ is extended to the whole space $X_{B}$ by setting $\widehat {\ov \mu} (X_B \setminus \widehat{X}_{\ol B}) = 0$.

Specifically, take a finite path $\ol e \in \ol E(v_0, v)$ from the top vertex $v_0$ to a vertex $v$ that belongs to the subdiagram $\ol B$.  Let $[\ol e]$ denote the cylinder subset of $X_{\ol B}$ determined by $\ol e$. For any finite path $\ol f \in E(v_0, v)$ from the diagram $B$ with the same range $v$, we set $\widehat{\ov \mu} ([\ol f])  = \ov \mu([\ol e])$. In such a way, the measure $\widehat{\ov \mu}$ is extended to the  $\sigma$-algebra of Borel subsets of $\wh X_{\ol B}$ generated by all clopen sets of the form  $[\ol z]$ where a finite path $\ol z$ has the range in a vertex from $\ol B$ and the tail of $z$ remains in $\ol B$. By construction, $\wh {\ov \mu}$  is $\mathcal E$-invariant and its restriction on $X_{\ol B}$ coincides with $\ov \mu$. We note that  the value $\wh{\ov \mu}(\wh X_{\ol B})$ can be either finite or infinite depending on the structure of $\ol B$ and $B$ (see (\ref{extension_method}) and Theorems \ref{neccsuff from BKMS}, \ref{edge_fin_krit} below). Furthermore, the set $\wh X_{\ol B}$  is said to be the {\em support} of $\widehat{\ov \mu}$.

Denote by $\wh X_{\ol B}^{(n)}$ the set of all paths $x = (x_i)_{i = 1}^{\infty}$  from $X_B$ such that $(x_1, \ldots, x_n)$ is a finite path in $B$ which starts at the vertex $v_0$ and ends at a vertex $v$ of $\ov B$, and the tail  $(x_{n+1},x_{n+2},\ldots)$ belongs to $\overline{B}$.
For instance, for a vertex subdiagram $\ov B$ we have
\begin{equation}\label{n-th level}
\wh X_{\ol B}^{(n)} = \{x = (x_i)\in \wh X_{\ol B} : r(x_i) \in W_i, \ \forall i \geq n\}.
\end{equation}
It is obvious that $\wh X_{\ol B}^{(n)} \subset \wh X_{\ol B}^{(n+1)}$ and
\begin{equation}\label{extension_method}
\widehat{\ov \mu}(\wh X_{\ol B}) = \lim_{n\to\infty} \widehat{\ov \mu}(\wh X_{\ol B}^{(n)}).
\end{equation}
More precisely, if $\ov B$ is a vertex subdiagram, then $\widehat{\ov \mu}(\wh X_{\ol B}^{(n)}) = \sum_{w\in W_n}  h^{(n)}_w \ov p^{(n)}_w$ and therefore
\begin{equation}\label{extension_method_vertex}
\widehat{\ov \mu}(\wh X_{\ol B}) = \lim_{n\to\infty}\sum_{w\in W_n}  h^{(n)}_w \ov p^{(n)}_w.
\end{equation}

In the case of an edge subdiagram $\ol B$,  the vertex set of $\ol B$ at level $n$ is $V_n$ and we obtain  a slightly different formula
\begin{equation}\label{extension_method_edge}
\widehat{\ov \mu}(\wh X_{\ol B}) = \lim_{n\to\infty}\sum_{w\in V_n}  h^{(n)}_w \ov p^{(n)}_w.
\end{equation}

\subsection{Some classes of Bratteli diagrams}\label{classes}

A Bratteli diagram $B = (V,E)$ is called of {\em finite rank} if there exists some natural number $d$ such that $|V_n| \leq d$ for every $n$. In a particular case, when all incidence matrices of $B$ are the same, the diagram $B$ is called {\em stationary}. If $B$ is a finite rank diagram, then it can be  telescoped to a diagram $B'$ where $B'$ has the same vertex set at each level.
Finite rank (simple or non-simple) Bratteli diagrams represent examples of Cantor systems (minimal or non-minimal) that have finitely many ergodic measures. These measures may be finite or infinite. Moreover, for each such a measure $\mu$, one can explicitly describe its support, a vertex subdiagram $B_\mu$ (see details and more results in  \cite{BKMS_2013}).

As in a recent work of Handelman \cite{H2013}, we also consider Bratteli diagrams whose incidence matrices satisfy the  {\em equal row sum} ($ERS$) and {\em equal column sum} ($ECS$)  properties.  It is said that a non-negative integer matrix $F = (f_{i,j})$ satisfies the equal row sum property ($F\in ERS$ or $F\in ERS(r)$) if $\sum_{j} f_{i,j} = r$ for all $i$ for some $r \in \mathbb{N}$.
A matrix $F = (f_{i,j})$ has the {\em equal column sum} property ($F \in ECS$) if the sum of entries of each column is the same. We write $ECS(c)$ if this value is  $c \in \mathbb N$.

For a Bratteli diagram $B$ defined by a sequence of incidence matrices $\{F_n\}_{n=0}^\infty$,
we write $B \in ERS(r_n)$ if $\sum_{w \in V_n} f^{(n)}_{v,w} = r_n$ for every $v \in V_{n+1}$
 and every $n$, and additionally $F_0 = h^{(1)} = (r_0, \ldots, r_0)^T$. It follows from (\ref{formula for heights}) that in this case $h^{(n)}_w = r_0 \cdots r_{n-1}$ for every $w \in V_n$.
Thus, for every probability $\mathcal{E}$-invariant measure on $B$, we have
$$
\sum_{w \in V_n} p_w^{(n)} = \frac{1}{r_0 \cdots r_{n-1}}
$$
for $n = 1, 2, \ldots$ The proof follows easily from (\ref{formula for measures}) by induction.
Similarly, we write $B \in ECS (c_n)$ if for every $w \in V_{n}$ and every $n \geq 1$ we have $\sum_{v \in V_{n+1}} f^{(n)}_{vw} = c_n$. Suppose that  $F_0 = (r_0, \ldots, r_0)^T$, then $c_0 = r_0 |V_1|$. According to (\ref{formula for measures}), we see that there exists a probability $\mathcal E$-invariant measure $\mu$ on $B$ defined by relations $p_v^{(n)} = (c_0\cdots c_{n-1})^{-1}$ where $v\in V_n$, $n\geq 1$.

If $F_n \in ERS(r_n) \cap ECS(c_n)$ for all $n$, then the fact that $\mu$ is a probability measure implies the relation $r_0\cdots r_{n-1}|V_n| = c_0\cdots c_{n-1}$.

\section{Finiteness of measure extension}\label{fineteness}

In this section, we focus on the following problem: given a Bratteli diagram $B$ and a subdiagram $\ol B$ of $B$, find necessary and sufficient conditions for finiteness of  measure extension $\wh { \ov \mu}(\wh X_{\ol B})$ when a probability measure $\ov \mu$ is defined on $\ol B$ of $B$. We prove several criteria here. We deal with both edge and vertex subdiagrams.

If $(F_n)$ is a  sequence of incidence matrices  of a Bratteli diagram $B$, then we can also define the sequence of stochastic matrices $(Q_n)$ with entries
$$
q^{(n)}_{v,w} = f^{(n)}_{v,w}\frac{h_w^{(n)}}{h_v^{(n+1)}},\ \ v\in V_{n+1},\ w\in V_{n}.
$$

For convenience of the reader, we collect here some results proved in \cite{BKK_14}. The following theorem gives criteria for finiteness of the measure extension.

\begin{thm}\label{thm from BKK}
Let $B$ be a Bratteli diagram with the sequence of incidence matrices $\{F_n\}_{n=0}^\infty$, and let $\ov B$ be a vertex subdiagram of $B$ defined by the sequence of subsets  $(W_n)_{n = 0}^\infty, \ W_n \subset V_n$. Suppose that $\ov \mu$ is  a probability  $\ol{\mathcal E}$-invariant measure on $X_{\ov B}$. Then the following properties are equivalent:
\begin{eqnarray*}
\widehat{\overline{\mu}}(\wh X_{\ov B}) < \infty & \Longleftrightarrow & \sum_{n=1}^\infty \sum_{v \in W_{n+1}} \sum_{w \in W'_n} f_{v,w}^{(n)} h_w^{(n)} \ov p_v^{(n+1)} < \infty\\
& \Longleftrightarrow & \sum_{n=1}^{\infty} \sum_{w \in W_{n+1}} \wh{\ov\mu}(X_{w}^{(n+1)})\sum_{v \in W^{'}_{n}} q_{w,v}^{(n)}
< \infty\\
& \Longleftrightarrow & \sum_{i=1}^{\infty} \left(\sum_{w\in W_{i+1}} h_w^{(i+1)} \ov p_w^{(i+1)} - \sum_{w\in W_{i}} h_w^{(i)} \ov p_w^{(i)}\right)< \infty.
\end{eqnarray*}
\end{thm}

The next  theorem contains some necessary and sufficient conditions for finiteness of the measure extension.

\begin{thm}\label{neccsuff from BKMS}
Let $B$ be a Bratteli diagram with incidence stochastic matrices $\{Q_n = (q_{v,w}^{(n)})\}$; and let $\ol B$ be a proper vertex subdiagram of $B$ defined by a sequence of subsets $(W_n)$ where $W_n \subset V_n$.

(1) Suppose that, for a probability invariant measure $\ov \mu$ on the path space $X_{\ol B}$, the extension $\widehat{\ov \mu}$ of $\ov \mu$ on  $\wh X_{\ol B}$ is finite. Then

(a)
\begin{equation}\label{neces2}
\sum_{n=1}^{\infty} \sum_{v \in W_{n+1}} \sum_{w \in W^{'}_{n}} q_{v,w}^{(n)} \mu(\ol X_{v}^{(n+1)})< \infty,
\end{equation}

(b)
$$
\sum_{n=1}^\infty \min_{v \in W_{n+1}}\max_{w \in W'_n} q_{v,w}^{(n)} <
\infty.
$$

(2) If
\begin{equation}\label{suffic}
\sum_{n=1}^{\infty} \sum_{v \in W_{n+1}} \sum_{w \in W^{'}_{n}} q_{v,w}^{(n)} < \infty,
\end{equation}
then any probability measure $\ov \mu$ defined on the path space $X_{\ol B}$ of the subdiagram $\ol B$ extends to a finite measure $\widehat{\ov \mu}$ on $\wh X_{\ol B}$.

(3) If
\begin{equation}\label{series I}
I = \sum_{n=1}^\infty \max_{v \in W_{n+1}} \left(\sum_{w \in W'_n} q_{vw}^{(n)}\right) < \infty,
\end{equation}
then, for any probability measure $\ov \mu$ defined on the path space $X_{\ol B}$ of the subdiagram $\ol B$, we have $\wh{\ov \mu}(\wh X_{\ol B}) < \infty$.

\end{thm}

In order to illustrate the methods used in the proof of Theorems \ref{thm from BKK} and \ref{neccsuff from BKMS}, we prove here condition (\ref{series I}).

\begin{proof} By Theorem \ref{thm from BKK} , it suffices to show that $I < \infty$ implies $S < \infty$ where
$$
S =\sum_{n=1}^\infty \sum_{v \in W_{n+1}} \sum_{w \in W'_n} f_{v,w}^{(n)} h_w^{(n)} \ov p_v^{(n+1)}.
$$
 We have
\begin{eqnarray*}
 \sum_{v \in W_{n+1}} \sum_{w \in W'_n} f_{v,w}^{(n)} h_w^{(n)} \ov p_v^{(n+1)} &=& \sum_{v \in W_{n+1}} \sum_{w \in W'_n} f_{v,w}^{(n)} h_w^{(n)} \frac{\ov p_v^{(n+1)} \ov h_v^{(n+1)}}{\ov h_v^{(n+1)}} \\
 &=& \sum_{v \in W_{n+1}} \ov \mu(\ov X_v^{(n+1)}) \sum_{w \in W'_n} f_{v,w}^{(n)} \frac{h_w^{(n)}}{\ov h_v^{(n+1)}}  \\
   & \leq & \max_{v \in W_{n+1}} \left(  \sum_{w \in W'_n} f_{v,w}^{(n)} \frac{h_w^{(n)}}{\ov h_v^{(n+1)}} \right) \sum_{v \in W_{n+1}} \ov \mu(\ov X_v^{(n+1)}).
\end{eqnarray*}

Since $\ov \mu$ is a probability measure, we obtain $\sum_{v \in W_{n+1}} \ov \mu(\ov X_v^{(n+1)}) = 1$.

We show that there exists $M > 0$ such that $\dfrac{h_w^{(n)}}{\ov h_w^{(n)}} < M$ for all $w \in W_n$ and all sufficiently large $n$.
Indeed, set $M_n = \max_{w \in W_n} \dfrac{h_w^{(n)}}{\ov h_w^{(n)}}$. Fix any $v \in W_{n+1}$. Then

\begin{eqnarray*}
\frac{h_v^{(n+1)}}{\ov h_v^{(n+1)}} &= &\frac{1}{\ov h_v^{(n+1)}}\left(\sum_{w\in W_n} f_{v,w}^{(n)} h_w^{(n)} + \sum_{w\in W'_n} f_{v,w}^{(n)} h_w^{(n)}\right)\\
&\leq &\frac{M_n}{\ov h_v^{(n+1)}}  \sum_{w\in W_n} f_{v,w}^{(n)} \ov h_w^{(n)} + \frac{1}{\ov h_v^{(n+1)}}\sum_{w\in W'_n} f_{v,w}^{(n)} h_w^{(n)}\\
&= & M_n + \frac{h_v^{(n+1)}}{\ov h_v^{(n+1)}}\sum_{w\in W'_n} f_{v,w}^{(n)} \frac{h_w^{(n)}}{h_v^{(n+1)}}\\
&= & M_n + \frac{h_v^{(n+1)}}{\ov h_v^{(n+1)}}\sum_{w\in W'_n} q_{v,w}^{(n)}\\
&\leq & M_n + \frac{h_v^{(n+1)}}{\ov h_v^{(n+1)}}\varepsilon_n,
\end{eqnarray*}
where
$$
\varepsilon_n = \max_{v \in W_{n+1}} \left(\sum_{w \in W'_n} q_{v,w}^{(n)}\right).
$$
Since $I < \infty$,  $\varepsilon_n \to 0$ as $n \to \infty$. From the above inequalities we obtain
$$
\frac{h_v^{(n+1)}}{\ov h_v^{(n+1)}} (1 - \varepsilon_n) \leq M_n \  \mbox{ and } \ M_{n+1} \leq \frac{M_n}{1 - \varepsilon_n}.
$$
Finally,
$$
M_n \leq \frac{M_1}{\prod_{k=1}^{\infty}(1 - \varepsilon_n)} =: M.
$$
Since $I < \infty$, we get that $M$ is finite.

Thus,
\begin{eqnarray*}
  \sum_{v \in W_{n+1}} \sum_{w \in W'_n} f_{v,w}^{(n)} h_w^{(n)} \ov p_v^{(n+1)}
&< & M \max_{v \in W_{n+1}} \left(  \sum_{w \in W'_n} f_{v,w}^{(n)}
 \frac{h_w^{(n)}}{h_v^{(n+1)}} \right)\\
 &= & M \max_{v \in W_{n+1}} \left(  \sum_{w \in W'_n} q_{v,w}^{(n)} \right).
\end{eqnarray*}
This completes the proof.
\end{proof}

It was also shown in \cite{BKK_14} that, in general, the sufficient condition (\ref{suffic}) is not necessary and the necessary condition  (\ref{neces2}) is not sufficient.

\begin{example}\label{ex1} The following example illustrates our approach to the study of measures on subdiagrams of a Bratteli diagram. In this example we answer the main problems we are interested in.

Let $B$ be a Bratteli diagram defined by rectangular incidence matrices
$$
F_n =
\begin{pmatrix}
1 & \ldots & 1 \\
\vdots & \ddots & \vdots\\
1 & \ldots & 1
\end{pmatrix}
$$
of size $|V_{n+1}|\times |V_n|$.
We remark that this type of Bratteli diagrams appears, for instance, for the symbolic minimal system known as the Grillenberger flow (see~\cite{G72}). Let $r_n = |V_n|$ and $c_n = |V_{n+1}|$, then $F_n \in ECS(c_n) \cap ERS(r_n)$, i.e.,  the sum of entries in each row is $r_n$ and the sum of entries in each column is $c_n$.

There is a unique probability invariant measure $\mu$ on $X_B$. Indeed, by~(\ref{formula for measures}), we have $p_w^{(n)} = p_{w'}^{(n)}$ for every $w, w' \in V_n$, and
$$
p_w^{(n)} = \frac{1}{|V_0| \cdots |V_{n}|}
$$
for every $w \in V_n$.

Let now $\ov B$ be a vertex subdiagram of $B$ defined by a sequence of vertices $(W_n)$ where $W_n \subsetneq V_n$. Compute $\mu(X_{\ov B})$ as follows:
\begin{eqnarray*}
\mu(X_{\ov B})  &=& \lim_{n \rightarrow \infty} \sum_{w \in W_n} p_w^{(n)} \ov h_{w}^{(n)} \\
  &=& \lim_{n \rightarrow \infty} \sum_{w \in W_n} \frac{|W_1| \ldots |W_{n-1}|}{|V_1| \cdots |V_n|}  \\
   &=& \lim_{n \rightarrow \infty} \prod_{i=1}^n \frac{|W_i|}{|V_i|}.
\end{eqnarray*}
Hence,
$$
\left(\mu(X_{\ov B}) > 0\right) \ \Longleftrightarrow \ \left(\prod_{i=1}^\infty \frac{|W_i|}{|V_i|} > 0\right) \ \Longleftrightarrow \ \left(\sum_{i=1}^\infty \left(1 - \frac{|W_i|}{|V_i|}\right) < \infty\right).
$$

Let $\ov \mu$ be the unique invariant measure on $X_{\ov B}$. Then $\ov p_w^{(n)} = (|W_1| \cdots |W_n|)^{-1}$.We obtain
\begin{eqnarray*}
 \widehat{\overline{\mu}}(\wh X_{\ov B})  &=&\lim_{n \rightarrow \infty} \sum_{w \in W_n} h_w^{(n)} \ov p_w^{(n)}\\
 &=& \lim_{n \rightarrow \infty} |W_n| \frac{|V_1| \cdots|V_{n-1}|}{|W_1| \cdots |W_n|}\\
  &=& \prod_{i=1}^\infty \frac{|V_i|}{|W_i|}.
\end{eqnarray*}
Finally,
\begin{equation}\label{equivalence in example}
\left(\widehat{\overline{\mu}}(\wh X_{\ov B}) < \infty\right) \ \Longleftrightarrow\ \left(\prod_{i=1}^\infty \frac{|V_i|}{|W_i|} < \infty\right) \ \Longleftrightarrow\ \left(\mu(X_{\ov B}) > 0\right).
\end{equation}

It follows from (\ref{equivalence in example}) that, if $\mu(X_{\ov B}) = 0$ then  $\widehat{\overline{\mu}}(\wh X_{\ov B}) = \infty$. We will see below that this fact is a particular case of a general result proved in Theorem \ref{main} and Corollary~\ref{corol_infty}.

We remark that if the measure of $X_{\ov B}$ is positive, then $\dfrac{|V_i|}{|W_i|}
\rightarrow 1$ as $i \rightarrow \infty$. It follows that if  $\ol B$ is, in particular,
a subdiagram of $B$  of finite rank, then $\mu(X_{\ov B}) = 0$.

Continuing this example, we can apply  the necessary and sufficient conditions of the finiteness
of measure extension found above. By Theorem \ref{neces2}, if $\widehat{\overline{\mu}}(X_B) < \infty$, then
$$
S_1 = \sum_{n=1}^{\infty}\min_{w \in W_{n+1}}\max_{v \in W_n'}q_{w,v}^{(n)} < \infty.
$$
For the considered example, we have
$$
S_1 = \sum_{n=1}^{\infty}\min_{w \in W_{n+1}}\max_{v \in W_n'}f_{w,v}^{(n)}\frac{h_v^{(n)}}{h_w^{(n+1)}} = \sum_{n=1}^{\infty} \frac{1}{|V_n|}.
$$
In particular,  we conclude that if, for the diagram $B$, the sequence $(|V_n|)$ is not growing sufficiently fast, then the probability invariant measure $\mu$ cannot be obtained as an extension from some vertex subdiagram.

We finish this example by the following observation.
The set $X_B \setminus \wh X_{\ov B}$ consists of all paths that visit vertices $W'_n = V_n \setminus W_n$ for infinitely many $n$'s.
Since $W_n \varsubsetneq V_n$, the sequence $(W_n')$ defines a proper subdiagram of $B$. If $\mu(X_{\ov B}) > 0$, then $\prod_{n=1}^\infty \dfrac{|V_n|}{|W_n|} < \infty$. Since $\dfrac{|V_n|}{|W_n|} > 1$ for all $n$, we have $\prod_{k=1}^\infty \dfrac{|V_{n_k}|}{|W_{n_k}|} < \infty$ for any increasing subsequence $\{n_k\}$. This inequality implies that $\prod_{k=1}^\infty \dfrac{|W'_{n_k}|}{|V_{n_k}|} = 0$ for any increasing subsequence $\{n_k\}$.

\end{example}
\medskip

Now we consider the case of an edge subdiagram.

\begin{thm}\label{edge_fin_krit}
Let $\overline{B}$ be an edge subdiagram of a Bratteli diagram $B$. For a probability invariant measure $\overline{\mu}$ on $X_{\overline{B}}$, the extension $\widehat{\overline{\mu}} (\widehat{X}_{\overline{B}})$ is finite if and only if
$$
\sum_{n=1}^{\infty} \sum_{v \in V_{n+1}} \sum_{w \in V_{n}} \widetilde{f}_{v,w}^{(n)} h_w^{(n)}\overline{p}_v^{(n+1)} < \infty
$$
where $\widetilde{f}_{v,w} = {f}_{v,w} - \ol{f}_{v,w}$.
Moreover, the following equality holds:
\begin{eqnarray*}
\widehat{\overline{\mu}}(\widehat{X}_{\overline{B}})  &=& \widehat{\overline{\mu}}(\widehat{X}_{\ol B}^{(1)}) + \sum_{n=1}^\infty \sum_{v\in V_{n+1}} \sum_{w\in V_{n}} \widetilde{f}_{v,w}^{(n)} h_w^{(n)} \overline{p}_v^{(n+1)}.
\end{eqnarray*}
\end{thm}

\begin{proof}
We recall that, for a given probability measure $\overline{\mu}$ on $\overline{B}$, the relations $\wh X_{\ol B}^{(n)} = \bigcup_{w \in V_n} \widehat{X}_w^{(n)}$ and $\widehat{\overline{\mu}}
 (\wh X _{\ov B}) = \lim_{n \rightarrow \infty}
\widehat{\overline{\mu}} (\wh X_{\ov B}^{(n)})$ hold for any $n\in \mathbb N$ (see (\ref{n-th level})).

Denote
$$
R_{n} = \sum_{w \in V_{n}} \tl h_w^{(n)} \ov p_w^{(n)},
$$
where $\tl h_w^{(n)} = h_w^{(n)} - \ov h_w^{(n)}$. Then $\widehat{\overline{\mu}} (\wh X_{\ov B}^{(n)}) = \ov \mu (X_{\ov B}) + R_n = 1 + R_n$.
Since
$$
\tl h_v^{(n+1)} = \sum_{w \in V_n} \tl f_{v,w}^{(n)} h_w^{(n)} +  \sum_{w \in V_n} \ov f_{v,w}^{(n)} \tl h_w^{(n)},
$$
we have
\begin{equation}\label{formula}
R_{n+1} = \sum_{v \in V_{n+1}} \sum_{w \in V_n} \tl f_{v,w}^{(n)} h_w^{(n)} \ov p_v^{(n+1)} + \sum_{v \in V_{n+1}} \sum_{w \in V_n} \ov f_{v,w}^{(n)} \tl h_w^{(n)} \ov p_v^{(n+1)}.
\end{equation}
On the other hand, we can represent the second summand in (\ref{formula}) as follows:
$$
\sum_{v \in V_{n+1}} \sum_{w \in V_n} \ov f_{v,w}^{(n)} \tl h_w^{(n)} \ov p_v^{(n+1)} = \sum_{w \in V_n} \tl h_w^{(n)} \ov p_w^{(n)} = R_n.
$$
This gives the relation
$$
R_{n+1} = R_n + \sum_{v \in V_{n+1}} \sum_{w \in V_n} \tl f_{v,w}^{(n)} h_w^{(n)} \ov p_v^{(n+1)}.
$$
Because $1 + R_1 = \widehat{\overline{\mu}}(\widehat{X}_{\ol B}^{(1)})$, we finally conclude that
\begin{eqnarray*}
\widehat{\overline{\mu}}(\widehat{X}_{\overline{B}})
  &=& \widehat{\overline{\mu}}(\widehat{X}_{\ol B}^{(1)}) + \sum_{n=1}^\infty \sum_{v\in V_{n+1}} \sum_{w\in V_{n}} \widetilde{f}_{v,w}^{(n)} h_w^{(n)} \overline{p}_v^{(n+1)}.
\end{eqnarray*}
Since $\widehat{\overline{\mu}}(\widehat{X}_{\ol B}^{(1)}) <\infty$, the statement of the theorem follows.
\end{proof}

\begin{remar}
One can notice that Theorem~\ref{edge_fin_krit} implies Theorem~\ref{thm from BKK} in the following nonrigorous way. Set
$$
\ov f_{vw}^{(n)} =
\left\{ \begin{array}{ll}
                               f_{vw}^{(n)} &  \mbox{if} \ v \in W_{n+1}, w \in W_{n}\\
                               0 &  \mbox{otherwise}.
                                \end{array}\right.
$$
Let $\ov p_v^{(n+1)} = 0$ for any $v \in W_{n+1}'$. Then we have
$$
\sum_{n=1}^{\infty} \sum_{v \in V_{n+1}} \sum_{w \in V_{n}} \widetilde{f}_{v,w}^{(n)} h_w^{(n)}\overline{p}_v^{(n+1)} = \sum_{n=1}^\infty \sum_{v \in W_{n+1}} \sum_{w \in W'_n} f_{v,w}^{(n)} h_w^{(n)} \ov p_v^{(n+1)}.
$$
\end{remar}

\begin{example} Let $B$ be a Bratteli diagram and consider the edge subdiagram $\ol B$ of $B$ which is obtained by removing only one edge from each set $E_n$ of edges between levels $n-1$ and $n$ in the diagram $B$. In other words, we have for all $n\geq 1$
$$
\wt f^{(n)}_{v,w} = \left\{ \begin{array}{ll}
                               0 &  \mbox{if} \ (v, w) \neq (v_{n+1},w_n)\\
                               1 &  \mbox{if} \ (v, w) = (v_{n+1},w_n)
                                \end{array}\right.
$$
for some $(v_{n+1},w_n) \in V_{n+1} \times V_n$ depending on $n$.
It follows from Theorem \ref{edge_fin_krit} that
$$
(\widehat{\overline{\mu}} (\widehat{X}_{\overline{B}}) < \infty) \ \Longleftrightarrow \ \sum_{n=1}^{\infty}  h_{w_n}^{(n)}\overline{p}_{v_{n+1}}^{(n+1)} < \infty.
$$

\end{example}

\begin{example}\label{ERS,ECS}
Let $B$ be a Bratteli diagram whose incidence matrices $(F_n)$ satisfy $ERS(r_n)$:  $\sum_{w \in V_n} f^{(n)}_{v,w} = r_n$ for every $v \in V_{n+1}$ and every $n$. Suppose  that $\overline{B}$ is an edge subdiagram of $B$ with incidence matrices  $\overline{F_n} \in ECS$. Hence, there exists $\ol c_n >0$ such that  $\sum_{v \in V_{n+1}} \overline{f}^{(n)}_{v,w} = \overline{c}_n$ for every $w \in V_{n}$ and every $n$. Denote by $T_n := \{(v,w) : \widetilde{f}_{v,w}^{(n)} \neq 0, \ v \in V_{n+1}, w \in V_n\}$. The set $T_n$ consists of pairs of vertices such that an edge between them has been removed from the set $E_n$ to construct $\ol B$. According to the results of Subsection \ref{classes}, there is a finite measure $\overline{\mu}$ on $X_{\overline{B}}$ such that $\overline{p}_v^{(n)} = (\overline{c}_0 \cdots \overline{c}_{n-1})^{-1}$. Then, by Theorem~\ref{edge_fin_krit},
\begin{equation}\label{ERS,ECS,neccsuff}
(\widehat{\overline{\mu}}(\widehat{X}_{\overline{B}}) < \infty)\ \Longleftrightarrow \
\sum_{n=1}^\infty \frac{r_0 \ldots r_{n-1}}{\overline{c}_0 \cdots \overline{c}_{n}} \sum_{(v,w) \in T_n} \widetilde{f}_{v,w}^{(n)} < \infty.
\end{equation}

 In order to illustrate the considered above case, we take  a Bratteli diagram  $B$ of finite rank two with incidence matrices
$$
F_n =
\begin{pmatrix}
a_n & c_n\\
d_n & b_n
\end{pmatrix},
$$
where $a_n + c_n = d_n + b_n = r_n$ for every $n$.  Denote $s_n(1) = a_n + d_n$ and $s_n(2) = c_n + b_n$. Without loss of generality, assume that $s_n(1) > s_n(2)$. Then there exist integers $x,y \geq 0$ such that $(a_n - x) + (d_n - y) = s_n(2)$. Define a subdiagram $\ol B$ whose incidence matrices are
$$
\ov F_n =
\begin{pmatrix}
a_n - x & c_n\\
d_n - y& b_n
\end{pmatrix},\ \ n\geq 1.
$$
Then $\ov F_n$ satisfies the $ECS$ property with $\ov c_n = s_n(2)$. Obviously, we have
$$
\tl F_n =
\begin{pmatrix}
x & 0\\
y& 0
\end{pmatrix}.
$$
Thus, $\sum_{v \in V_{n+1}} \sum_{w \in V_{n}} \widetilde{f}_{v,w}^{(n)} = x + y = s_n(1) - \ov c_n$.
By Theorem~\ref{edge_fin_krit}, we obtain that
$$
(\widehat{\overline{\mu}}(\widehat{X}_{\overline{B}}) < \infty) \ \Longleftrightarrow  \ \sum_{n=1}^\infty \frac{r_0 \cdots r_{n-1}}{\overline{c}_0 \cdots \overline{c}_{n}} (s_n(1) - s_n(2)) < \infty.
$$
 \end{example}

We now return to the problem of finding conditions that imply finiteness of measure extensions.
Similarly to Theorem \ref{neccsuff from BKMS} (3), we  find a sufficient condition for finiteness of $\widehat{\overline{\mu}}(\wh X_{\ov B})$ for an edge subdiagram $\ol B$ and a probability invariant measure $\ol\mu$ on $X_{\ol B}$.

Denote
$$
J = \sum_{n=1}^\infty \sum_{v \in V_{n+1}} \sum_{w \in V_n} \tl f_{v,w}^{(n)} h_w^{(n)} \ov p_v^{(n+1)}.
$$
By Theorem \ref{edge_fin_krit}, $\widehat{\overline{\mu}}(\wh X_{\ov B})$ is finite if and only if $J < \infty$.

\begin{prop}
Let $B$, $\ov B$, $\ov \mu$, and $J$ be as above. If
$$
\sum_{n=1}^\infty \max_{v \in V_{n+1}} \left( \sum_{w \in V_n} \tl f_{v,w}^{(n)} \frac{h_w^{(n)}}{\ov h_v^{(n+1)}} \right) < \infty,
$$
then $\widehat{\overline{\mu}}(\wh X_{\ov B})$ is finite.
\end{prop}

\begin{proof} The general term of $J$ can be estimated as follows:
\begin{eqnarray*}
\sum_{v \in V_{n+1}} \sum_{w \in V_n} \tl f_{v,w}^{(n)} h_w^{(n)} \ov p_v^{(n+1)} &=& \sum_{v \in V_{n+1}} \sum_{w \in V_n} \tl f_{v,w}^{(n)} h_w^{(n)} \ov p_v^{(n+1)} \frac{\ov h_v^{(n+1)}}{\ov h_v^{(n+1)}}\\
 &=& \sum_{v \in V_{n+1}} \sum_{w \in V_n} \tl f_{v,w}^{(n)} h_w^{(n)} \ov \mu (\ov X_v^{(n+1)}) \frac{1}{\ov h_v^{(n+1)}}  \\
 &=&  \sum_{v \in V_{n+1}} \ov \mu (\ov X_v^{(n+1)}) \max_{v \in V_{n+1}} \left( \sum_{w \in V_n} \tl f_{v,w}^{(n)} \frac{h_w^{(n)}}{\ov h_v^{(n+1)}} \right).
\end{eqnarray*}
Since $\sum_{v \in V_{n+1}} \ov \mu (\ov X_v^{(n+1)}) = 1$, we obtain the desired result.
\end{proof}


\section{The number of finite ergodic measures for finite rank Bratteli diagrams}\label{number_of_erg}
We begin with considering in detail the class of  Bratteli diagrams of rank two that have the ERS property. It turns out that this class can be studied completely, and one can find necessary and sufficient conditions for a diagram to have either a single finite ergodic measure or two finite ergodic measures.

\begin{prop}\label{rank2example}
Let $B$ be a Bratteli diagram with $2 \times 2$ incidence matrices $F_n$ satisfying $ERS$:
$$
F_n =
\begin{pmatrix}
a_n & c_n\\
d_n & b_n
\end{pmatrix},
$$
where $a_n + c_n = d_n + b_n = r_n$ for every $n$.
Then

(1) There are exactly two finite ergodic invariant measures on $B$ if and only if
\begin{equation}\label{2measures}
\sum_{k=1}^{\infty}\left(1 - \frac{|a_{k} - d_{k}|}{r_k}\right) < \infty,
\end{equation}
or, equivalently,
\begin{equation}\label{2series}
\sum_{k=1}^{\infty}\left(1 - \frac{\max\{a_{k},d_{k}\}}{r_k}\right) < \infty\  \  \mbox{ and }
\  \  \sum_{k=1}^{\infty} \frac{\min\{a_{k},d_{k}\}}{r_k} < \infty.
\end{equation}
In this case, one can point out explicitly the subdiagrams (odometers) that support these measures.

(2) There is a unique invariant measure $\mu$ on $B$ if and only if
\begin{equation}\label{1measure}
\sum_{k=1}^{\infty}\left(1 - \frac{|a_{k} - d_{k}|}{r_k}\right) = \infty.
\end{equation}
Moreover, if
\begin{equation}\label{min}
\sum_{k=1}^{\infty} \min\left\{\frac{\min\{a_{k},d_{k}\}}{r_k}, 1 - \frac{\max\{a_{k},d_{k}\}}{r_k} \right\}= \infty,
\end{equation}
then there is no odometer such that the unique measure $\mu$ would be the extension of a measure supported by this odometer. Otherwise, if for instance   (\ref{1measure}) holds and (\ref{min}) does not hold, then there is an example when the unique invariant measure is an extension from an odometer and there is an example when it is not.
\end{prop}


\begin{proof} Without loss of generality, we can assume that $r_0 =1$, that is the diagram has single edges between $V_0$ and $V_1$.

Let $\mu$ be any probability invariant measure on $B$.
Let $p_0^{(n)}$ and $p_1^{(n)}$ be the measures of cylinder sets of length $n$ that end at the vertices $v_0\in V_n$ and $v_1 \in V_n$, respectively. Then we have  for any $n \geq 1$
$$
p_0^{(n)} = a_n p_0^{(n+1)} + d_n p_1^{(n+1)},
$$
$$
p_1^{(n)} = c_n p_0^{(n+1)} + b_n p_1^{(n+1)}.
$$
 Clearly, $p_0^{(1)} + p_1^{(1)} = \dfrac{1}{r_0}  = 1$. It is easy to see that $h_i^{(n)} = r_0 \cdots r_{n-1}$, $i = 0,1$, and
$$
p_0^{(n)} + p_1^{(n)} = \frac{1}{r_0 \cdots r_{n-1}}
$$
for any $n \geq 1$.
Hence,
$$
p_0^{(n)} = a_n p_0^{(n+1)} + d_n \left(\frac{1}{r_0 \cdots r_{n}} - p_0^{(n+1)}\right) = (a_n - d_n) p_0^{(n+1)} + \frac{d_n}{r_0 \cdots r_{n}}.
$$
Therefore we have
\begin{equation}\label{composition}
p_0^{(n)} = \prod_{k=0}^m (a_{n+k} - d_{n+k})p_0^{(n+m+1)} + C_{n,m},
\end{equation}
where
$$
C_{n,m} = \frac{d_n}{r_0 \ldots r_n} + \sum_{k=0}^{m-1}(a_n - d_n)\cdots (a_{n+k} - d_{n+k})\frac{d_{n+k+1}}{r_0 \ldots r_{n+k+1}}
$$
for $m = 1,2, \ldots$ In particular,  $C_{n,0} = \dfrac{d_n}{r_0 \cdots r_n}$.

Thus, the measure $\mu$ is completely determined by the sequence of numbers $\{p_0^{(n)}\}$ such that $0 \leq p_0^{(n)} \leq \dfrac{1}{r_0 \cdots r_{n-1}}$ and
\begin{equation}\label{Gdef}
p_0^{(n)} = (a_n - d_n)p_0^{(n+1)} + \frac{d_n}{r_0 \cdots r_n}.
\end{equation}

Denote by $\Delta^{(n)}$ the interval $\left[0, \dfrac{1}{r_0 \cdots r_{n-1}}\right]$.
Then we have a sequence of maps
$$
\Delta^{(1)} \stackrel{G_1}{\longleftarrow} \Delta^{(2)} \stackrel{G_2}{\longleftarrow} \Delta^{(3)} \stackrel{G_3}{\longleftarrow} \ldots
$$
where the linear map $G_n$ is defined by (\ref{Gdef}). Relation (\ref{composition}) presents the composition
$$
G_n \circ \ldots \circ G_{n+m} \colon \Delta^{(n+m+1)} \rightarrow \Delta^{(n)}.
$$
We have $G_n \circ \ldots \circ G_{n+m}(0) = C_{n,m}$ and
$$
G_n \circ \ldots \circ G_{n+m}\left(\frac{1}{r_0 \cdots r_{n+m}}\right) = \frac{1}{r_0 \cdots r_{n-1}}\prod_{k=0}^m \frac{(a_{n+k} - d_{n+k})}{r_{n+k}} + C_{n,m}.
$$
Thus, since  $G_n$ is a linear map in variable $p_0^{(n+1)}$, we obtain
$$
|G_n \circ \ldots \circ G_{n+m}(\Delta^{(n+m+1)})| = \frac{1}{r_0 \cdots r_{n-1}}\prod_{k=0}^m \frac{|a_{n+k} - d_{n+k}|}{r_{n+k}}
$$
where $|\Delta|$ stands for the length of an interval $\Delta$.
It follows from this relation  that there is a unique measure $\mu$ on $B$ if and only if
$$
\prod_{k=0}^\infty \frac{|a_{k} - d_{k}|}{r_k} = 0
$$
or, equivalently,
$$
\sum_{k=1}^{\infty}\left(1 - \frac{|a_{k} - d_{k}|}{r_k}\right) = \infty.
$$

On the other hand, the diagram $B$ has two ergodic measures on $B$ if and only if
$$
\sum_{k=1}^{\infty}\left(1 - \frac{|a_{k} - d_{k}|}{r_k}\right) < \infty.
$$
We notice that
$$
\sum_{k=1}^{\infty}\left(1 - \frac{|a_{k} - d_{k}|}{r_k}\right) = \sum_{k=1}^{\infty}\left(1 - \frac{\max\{a_{k},d_{k}\}}{r_k}\right) + \sum_{k=1}^{\infty} \frac{\min\{a_{k},d_{k}\}}{r_k}.
$$
Therefore
$$
\sum_{k=1}^{\infty}\left(1 - \frac{|a_{k} - d_{k}|}{r_k}\right) < \infty
$$
if and only if
\begin{equation*}
\sum_{k=1}^{\infty}\left(1 - \frac{\max\{a_{k},d_{k}\}}{r_k}\right) < \infty\  \  \mbox{ and }
\  \  \sum_{k=1}^{\infty} \frac{\min\{a_{k},d_{k}\}}{r_k} < \infty.
\end{equation*}
In this case, each ergodic measure is the extension of a measure from an odometer. We show how to find these odometers. Suppose that (\ref{2series}) holds.

\vskip0.3cm
\textbf{\textit{Claim}}. A measure $\mu$ is an extension from an odometer $\ov B (W_n)$ if and only if
\begin{equation}\label{short}
\sum_{n=1}^{\infty} \frac{f_{w_{n+1},v_n'}^{(n)}}{r_n} < \infty,
\end{equation}
where $W^{'}_n = \{v^{'}_n\}$ and $W_{n+1} = \{w_{n+1}\}$.
\vskip0.3cm

\textit{Proof of the claim}. It is easy to see that in our case
\begin{equation}\label{2odomext}
\sum_{n=1}^\infty \sum_{v \in W_{n+1}} \sum_{w \in W'_n} f_{v,w}^{(n)} h_w^{(n)} \ov p_v^{(n+1)} = \sum_{n=1}^\infty \frac{r_0 \cdots r_{n-1}}{\prod_{k=1}^n f_{w_{k+1},w_k}^{(k)}} f_{w_{n+1},v_n'}^{(n)}.
\end{equation}
By Theorem \ref{thm from BKK}, the measure extension $\wh{\ov \mu}$ is finite if and only if the series (\ref{2odomext}) converges.
We observe that
$$
\sum_{n=1}^\infty \frac{r_0 \cdots r_{n-1}}{\prod_{k=1}^n f_{w_{k+1},w_k}^{(k)}} f_{w_{n+1},v_n'}^{(n)} < \infty\  \Longleftrightarrow \ \sum_{n=1}^{\infty} \frac{f_{w_{n+1},v_n'}^{(n)}}{r_n} < \infty.
$$
Indeed, since $\dfrac{r_0 \cdots r_{n-1}}{\prod_{k=1}^n f_{w_{k+1},w_k}^{(k)}} > 1$, the direction ``$\Longrightarrow$'' is clear. To prove ``$\Longleftarrow$'', we use equality
$$
1 - \frac{f_{w_{n+1},v'_n}^{(n)}}{r_n} = \frac{f_{w_{n+1},v'_n}^{(n)}}{r_n}.
$$
Hence,
$$
\sum_{n=1}^\infty \frac{f_{w_{n+1},v_n'}^{(n)}}{r_n} < \infty\ \Longleftrightarrow \ \prod_{n=1}^\infty \left(1 - \frac{f_{w_{n+1},v_n'}^{(n)}}{r_n} \right) > 0 \ \Longleftrightarrow \ \prod_{n=1}^\infty \frac{f_{w_{n+1},v'_n}^{(n)}}{r_n} > 0.
$$
Therefore, $\prod_{n=1}^\infty \dfrac{r_n}{f_{w_{n+1},v'_n}^{(n)}} < \infty$ and there exists $K > 0$ such that $\prod_{n=1}^{N} \dfrac{r_n}{f_{w_{n+1},v'_n}^{(n)}} < K$ for every $N$, and the claim is proved.
\vskip0.3cm

Thus, if (\ref{2series}) holds we have
$$
\sum_{k=1}^{\infty} \frac{\min\{f^{(k)}_{11},f^{(k)}_{21}\}}{r_k} < \infty.
$$
Notice that, by the ERS property, the equality $(1 - \max\{a_{k},d_{k}\}) = \min\{b_{k},c_{k}\}$ holds.
Hence if $f^{(k)}_{11} < f^{(k)}_{21}$ then $f^{(k)}_{22} < f^{(k)}_{12}$. Therefore, in order to obtain (\ref{short}), one of the odometers should go through vertices $\{w_1^{(k)}, w_2^{(k+1)}\}$ and the other through vertices $\{w_2^{(k)}, w_1^{(k+1)}\}$. Otherwise, these odometers go through vertices $\{w_1^{(k)}, w_1^{(k+1)}\}$ and $\{w_2^{(k)}, w_2^{(k+1)}\}$. Thus, for each two consecutive levels, we define two disjoint sets of vertices. The choice of one of these sets for the first and second level uniquely defines $\ov B (W_n)$. Indeed, suppose that $f^{(1)}_{11} < f^{(1)}_{21}$ and we choose $\{w_1^{(1)},  w_2^{(2)}\}$. Then to form a subdiagram, the next set of vertices should contain $w_2^{(2)}$ and so on. We observe that in this case we do not need the procedure of telescoping to find subdiagrams supporting finite ergodic measures.

Now we prove the last part of the theorem. Denote
\begin{eqnarray*}
m_k &=& \min\left\{\frac{\min\{a_{k},d_{k}\}}{r_k}, 1 - \frac{\max\{a_{k},d_{k}\}}{r_k} \right\}\\
&=& \frac{\min\{a_{k}, b_{k}, c_{k}, d_{k}\}}{r_k}.
\end{eqnarray*}
Suppose that $\mu$ is an extension of a measure $\ov \mu$ from an odometer defined by vertices $(w_n)$.
We obtain
\begin{equation}\label{m_n}
\sum_{n=1}^{\infty} \sum_{v \in W_{n+1}} \sum_{w \in W_n'} \frac{f_{v,w}^{(n)}}{r_n} \geq \sum_{n=1}^{\infty} m_n = \infty.
\end{equation}
As it was proved above, a measure is an extension from an odometer $\ov B (W_n)$ if and only if
Eq.~(\ref{short}) holds.
Hence, inequality (\ref{m_n}) implies that that the extension of a measure from any odometer is infinite.

For the case when~(\ref{1measure}) holds and~(\ref{min}) does not hold, the examples of measures are provided in the next series of examples.
\end{proof}


The following example provides us with a diagram $B$ and unique ergodic invariant measure $\mu$ such that (\ref{1measure}) holds, (\ref{min}) does not hold, and $\mu$ is an extension from an odometer.

\begin{example}\label{infty}

Consider a Bratteli diagram $B$ of rank two with incidence matrices
$$
F_{2n+1} =
\begin{pmatrix}
2 & a_n\\
\dfrac{a_n}{2} + 1 & \dfrac{a_n}{2} + 1
\end{pmatrix},
$$
where $a_n$ is an even number and
$$
\sum_{n=1}^\infty \frac{1}{a_n} < \infty.
$$
Let
$$
F_{2n} =
\begin{pmatrix}
\dfrac{a_n}{2} + 1 & \dfrac{a_n}{2} + 1\\
a_n & 2
\end{pmatrix}.
$$
All matrices $F_n$ have the property ERS with row sum $r_{2n} = r_{2n + 1} = a_n + 2$. Consider the series (\ref{2series}) for this example. We have
$$
\frac{\min\{f_{11}^{(2n)}, f_{21}^{(2n)}\}}{a_n + 2} = \frac{1}{2}; \quad 1 - \frac{\max\{f_{11}^{(2n)}, f_{21}^{(2n)}\}}{a_n + 2} = \frac{2}{a_n + 2}.
$$
Also
$$
\frac{\min\{f_{11}^{(2n+1)}, f_{21}^{(2n+1)}\}}{a_n + 2} = \frac{2}{a_n + 2}; \quad 1 - \frac{\max\{f_{11}^{(2n+1)}, f_{21}^{(2n+1)}\}}{a_n + 2} = \frac{1}{2}.
$$
Hence, both of the series~(\ref{2series}) diverge, but the series~(\ref{min}) converges. By Proposition~\ref{rank2example}, there is a unique invariant probability measure on $X_B$. Consider a vertex subdiagram $\ov B$ of $B$ such that $W_n$ consists of the second vertex for $n$ odd and of the first vertex for $n$ even. Therefore, the incidence matrices for $\ov B$ are $\ov F_{2n} = \ov F_{2n+1} = a_n$. Let $\ov \mu$ be the probability invariant measure on $\ov B$. Then $\wh{\ov \mu}(\wh X_{\ov B})$ is finite. Indeed, it suffices to check that
$$
\sum_{n=1}^{\infty} \frac{f_{w_{n+1},v_n'}^{(n)}}{r_n} =  \frac{2}{a_0 + 2} + 2 \sum_{n=1}^{\infty} \frac{2}{a_n + 2} < \infty
$$
and apply Proposition \ref{rank2example}. Hence, the unique measure $\mu$ is an extension of measure $\ov \mu$ from odometer $X_{\ov B}$. Thus, both series in~(\ref{2series}) diverge, but the measure $\mu$ is not an extension from some odometer.
\end{example}

\begin{remar}
Example~\ref{infty} can be slightly modified to obtain a situation when one of the series in~(\ref{2series}) diverges and the other one converges. For instance, let
$$
F_{n} =
\begin{pmatrix}
a_n & 2\\
\dfrac{a_n}{2} + 1 & \dfrac{a_n}{2} + 1
\end{pmatrix}
$$
 for every $n$. Set $W_n$ to be the first vertex on each level. Then the measure $\mu$ is an extension of a measure from $X_{\ov B}$; in this particular situation, we do not need the procedure of telescoping to find a subdiagram $\ov B$.
\end{remar}

In the next example, we provide a diagram $B$ and unique ergodic invariant measure $\mu$ such that~(\ref{min}) does not hold and $\mu$ is not an extension from any odometer.
\begin{example}
Suppose now that
$$
F_{n} =
\begin{pmatrix}
2 & a_n\\
\dfrac{a_n}{2} + 1 & \dfrac{a_n}{2} + 1
\end{pmatrix}.
$$
Then there is a unique ergodic invariant measure $\mu$ which is not an extension from some odometer. Indeed, suppose that there is an odometer $\ov B$ sitting on $ \{w_n\}$ and supporting $\mu$. Then
\begin{equation*}
\sum_{n=1}^{\infty} \frac{f_{w_{n+1},w_n'}^{(n)}}{r_n} < \infty.
\end{equation*}
Notice that
$$
\frac{f_{12}^{(n)}}{r_n} > \frac{1}{2} \mbox{ and } \frac{f_{21}^{(n)}}{r_n} = \frac{f_{22}^{(n)}}{r_n} = \frac{1}{2}
$$
for every $n$. Thus, series $\sum_{n=1}^{\infty} \frac{f_{w_{n+1},w_n'}^{(n)}}{r_n}$ should have infinitely many elements $\frac{f_{11}^{(n)}}{r_n}$ and finitely many others, but that's obviously impossible. Thus, even in case when one of the series in~(\ref{2series}) converges and the other diverges, the measure might not be an extension from some odometer.
\end{example}

\begin{remar}
In the previous example, set $a_n = 2^n$ and telescope the diagram with respect to odd levels. Then the new incidence matrices must be
$$
\begin{pmatrix}
2 & 2^n\\
2^{n-1} + 1 & 2^{n-1} + 1
\end{pmatrix}
\begin{pmatrix}
2 & 2^{n-1}\\
2^{n-2} + 1 & 2^{n-2} + 1
\end{pmatrix}
=
$$
$$
\begin{pmatrix}
2^{2n-2}+2^n+4  & 2^{2n-2}+2^{n+1}\\
2^{2n - 3}+2^n+2^{n-1}+2^{n-2}+3 & 2^{2n - 2}+2^{2n-3}+2^{n}+2^{n-2}+1
\end{pmatrix}
$$
and the new row sum of the matrix is  $r_n r_{n-1} = 2^{2n - 1}+2^{n+1}+2^{n}+4$. We see that
$$
\sum_{k=1}^{\infty} \min\left\{\frac{\min\{a_{k},d_{k}\}}{r_k}, 1 - \frac{\max\{a_{k},d_{k}\}}{r_k} \right\}= \infty,
$$
and this proves that the measure is not an extension from some odometer. Besides we have shown that convergence of one of the series in~(\ref{2series}) is not preserved under telescoping. On the other hand,  the number of ergodic measures is  preserved, so that the conditions in Proposition \ref{rank2example} are not invariant under telescoping.

\end{remar}

\begin{example}
The examples given in \cite{BKMS_2013} (see Example 4.13 and Remark 5.9 there) illustrate our  Proposition \ref{rank2example}. In case when
$$
F_n =
\begin{pmatrix}
n^2 & 1\\
1 & n^2
\end{pmatrix},
$$
there are two finite ergodic invariant measures.
In case when
$$
F_n =
\begin{pmatrix}
n & 1\\
1 & n
\end{pmatrix},
$$
there is a unique invariant measure $\mu$ which is not an extension from any odometer.
\end{example}
\medskip

As known, any Bratteli diagram of rank $k$ can have at most $k$ ergodic measures. The next result gives a necessary and sufficient condition under which such a diagram has exactly $k$ ergodic measures.

\begin{thm}
Let $B = (V,E)$ be a Bratteli diagram of rank $k \geq 2$; identify $V_n$ with $\{1, ... ,k\}$ for any $n \geq 1$. Let $F_n = (f_{i,j}^{(n)})$ form a sequence of  incidence matrices of $B$ such that
$\sum_{j \in V_n} f_{i,j}^{(n)} = r_n \geq 2$ for every $i \in V_{n+1}$. Suppose that  $\rank\ F_n = k$ for all $n$. Denote
$$
z^{(n)} =
\det
\begin{pmatrix}
\dfrac{f_{1,1}^{(n)}}{r_n} & \ldots & \dfrac{f_{1,{k-1}}^{(n)}}{r_n} & 1\\
\vdots & \ddots & \vdots & \vdots\\
\dfrac{f_{{k},1}^{(n)}}{r_n} & \ldots & \dfrac{f_{{k},{k-1}}^{(n)}}{r_n} & 1
\end{pmatrix}.
$$
Then there exist exactly $k$ ergodic invariant measures on $B$ if and only if
$$
\prod_{n = 1}^{\infty} |z^{(n)}| > 0,
$$
or, equivalently,
$$
\sum_{n = 1}^{\infty}(1 - |z^{(n)}|) < \infty.
$$
\end{thm}

\begin{proof}
We will use the ideas from the proof of Proposition \ref{rank2example}.
Let $\mu$ be any probability invariant measure on $B$.
Recall that $p_i^{(n)}$ denotes the measure of a cylinder set of length $n$ that ends at the vertex
$i \in V_n$. Since the incidence matrices $F_n$ satisfy ERS, we observe that, for any $n \geq 1$,
$$
\sum_{i = 1}^{k} p_i^{(n)} = \frac{1}{r_0 \cdots r_{n-1}}.
$$
We also have
$$
p^{(n)} = F^T_n p^{(n+1)}
$$
for every $n \in \mathbb{N}$.


Denote
$$
\Delta^{(n)} := \{(x_1, \ldots, x_k)^T : \sum_{i = 1}^k x_i = \frac{1}{r_0 \cdots r_{n-1}}
\ \mbox{ and } \ x_i   \geq 0, \  1 \leq i \leq k \}.
$$
Then $\Delta^{(n)} \subset \mathbb{R}^{k}_{+}$ contains all possible values for  vectors
 $(p_1^{(n)}, \ldots, p_k^{(n)})$ corresponding to invariant probability measures on $B$.
Denote $G_n = F_n^{T}$. Then we have
$$
\Delta^{(1)} \stackrel{G_1}{\longleftarrow} \Delta^{(2)} \stackrel{G_2}{\longleftarrow} \Delta^{(3)} \stackrel{G_3}{\longleftarrow} \ldots
$$
Since $\rank\ F = k$, each $G_n$ is injective. Let
$$
\Delta_m^{(n)} = G_n \circ \ldots \circ G_{n+m-1}(\Delta^{(n+m)}).
$$
for $m = 1,2 \ldots$ Since each $F_n$ is a non-negative matrix
we have for any fixed $n$
$$
\Delta^{(n)} \supset \Delta_1^{(n)} \supset \Delta_2^{(n)} \supset \ldots
$$
For every $m \geq 1$, the mappings
$$
G_n \circ \ldots \circ G_{n + m - 1} \colon \Delta^{(n + m)} \longrightarrow \Delta_{m}^{(n)} \subset \Delta^{(n)}
$$
are one-to-one and onto. The sets $\Delta^{(n)}_m$ are simplices in $\mathbb{R}^{k-1}$ with vertices (considered as points in $\mathbb{R}^{k-1}$)
$\{G_n \circ \ldots \circ G_{n+m-1}(e_i^{(n+m)}) : i = 1, ... ,k\}$, where
$$
e_i^{(n)} = \left( 0,\ldots, 0, \frac{1}{r_0 \ldots r_{n-1}}, 0, \ldots, 0  \right)^{T}
$$
and the non-zero element corresponds to the $i$-th coordinate.
In particular,
$$
G_n\left(e_i^{(n+1)}\right) = \left(\frac{f^{(n)}_{i1}}{r_0 \ldots r_{n}}, \ldots, \frac{f^{(n)}_{ik}}{r_0 \ldots r_{n}} \right)^{T} \in \Delta^{(n)}.
$$
The intersection $\Delta^{(n)}_{\infty} = \bigcap_{m=1}^{\infty} \Delta_m^{(n)}$ is a simplex with at most $k$ vertices  (see \cite{phelps:2001, pullman:1971}). Hence $\Delta^{(n)}_{\infty}$ has $k$ vertices if and only if the $(k-1)$-dimensional Lebesgue measure ${\rm vol}_{k-1}(\Delta^{(n)}_{\infty})$ is positive.
We have ${\rm vol}_{k-1}(\Delta^{(n)}_{\infty}) = \lim_{m \rightarrow \infty} {\rm vol}_{k-1}(\Delta^{(n)}_{m})$.

Consider
\begin{multline*}
\frac{{\rm vol}_{k-1}(\Delta^{(n)}_{m})}{{\rm vol}_{k-1}(\Delta^{(n)})} = \frac{{\rm vol}_{k-1}[G_n \circ \ldots \circ G_{n+m}(\Delta^{(n+m+1)})]}{{\rm vol}_{k-1}[G_n \circ \ldots \circ G_{n+m-1}(\Delta^{(n+m)})]}\\
 \cdot \frac{{\rm vol}_{k-1}[G_n \circ \ldots \circ G_{n+m-1}(\Delta^{(n+m)})]}{{\rm vol}_{k-1}[G_n \circ \ldots \circ G_{n+m-2}(\Delta^{(n+m-1)})]} \cdot \ldots \cdot \frac{{\rm vol}_{k-1}[G_n (\Delta^{(n+1)})]}{{\rm vol}_{k-1}(\Delta^{(n)})}.
\end{multline*}
Since for $m = 0,1, \ldots$ and $n \geq 1$,
$$
\frac{{\rm vol}_{k-1}[G_n \circ \ldots \circ G_{n+m}(\Delta^{(n+m+1)})]}{{\rm vol}_{k-1}[G_n \circ \ldots \circ G_{n+m-1}(\Delta^{(n+m)})]} =  \frac{{\rm vol}_{k-1}[G_{n+m} (\Delta^{(n+m+1)})]}{{\rm vol}_{k-1}(\Delta^{(n+m)})}
$$
 we have
$$
\frac{{\rm vol}_{k-1}[\Delta^{(n)}_{m}]}{{\rm vol}_{k-1}[\Delta^{(n)}]} = \frac{{\rm vol}_{k-1}[G_n(\Delta^{(n+1)})]}{{\rm vol}_{k-1}[\Delta^{(n)}]} \cdot \ldots \cdot \frac{{\rm vol}_{k-1}[G_{n+m}(\Delta^{(n+m+1)})]}{{\rm vol}_{k-1}[\Delta^{(n+m)}]}.
$$
On the other hand, it can be easily seen that
$$
{\rm vol}_{k-1}(\Delta^{(n)}) = \frac{1}{(k-1)!}\left(\frac{1}{r_0 \cdots r_{n-1}} \right)^{k-1}
$$
and
$$
{\rm vol}_{k-1}(\Delta^{(n)}_1) = \frac{1}{(k-1)!} \left|\det
\begin{pmatrix}
\dfrac{f_{1,1}^{(n)}}{r_0 \cdots r_n} & \ldots & \dfrac{f_{1, k-1}^{(n)}}{r_0 \cdots r_n} & 1\\
\vdots & \ddots & \vdots & \vdots\\
\dfrac{f_{k,1}^{(n)}}{r_0 \cdots r_n} & \ldots & \dfrac{f_{k, k-1}^{(n)}}{r_0 \cdots r_n} & 1
\end{pmatrix}\right|.
$$
Therefore, we obtain that
$$
\frac{{\rm vol}_{k-1}[G_{n+m} (\Delta^{(n+m+1)})]}{{\rm vol}_{k-1}[\Delta^{(n+m)}]} = \frac{{\rm vol}_{k-1}[(\Delta^{(n+m)}_1)]}{{\rm vol}_{k-1}[\Delta^{(n+m)}]} = \left|\det
\begin{pmatrix}
\dfrac{f_{1,1}^{(n+m)}}{r_{n+m}} & \ldots & \dfrac{f_{1,{k-1}}^{(n+m)}}{r_{n+m}} & 1\\
\vdots & \ddots & \vdots & \vdots\\
\dfrac{f_{{k},1}^{(n+m)}}{r_{n+m}} & \ldots & \dfrac{f_{{k},{k-1}}^{(n+m)}}{r_{n+m}} & 1
\end{pmatrix}\right|.
$$
for $m = 0,1, \ldots$ Finally, the following formula is deduced from the above relations.
$$
\frac{{\rm vol}_{k-1}[(\Delta^{(n)}_m)]}{{\rm vol}_{k-1}[\Delta^{(n)}]} = \prod_{s = n}^{n + m} \left|\det
\begin{pmatrix}
 \dfrac{f_{1,1}^{(s)}}{r_{s}} & \ldots & \dfrac{f_{1,{k-1}}^{(s)}}{r_{s}} & 1\\
\vdots & \ddots & \vdots & \vdots\\
\dfrac{f_{{k},1}^{(s)}}{r_{s}} & \ldots & \dfrac{f_{{k},{k-1}}^{(s)}}{r_{s}} & 1
\end{pmatrix}\right| =
\prod_{s = n}^{n + m} |z^{(s)}|.
$$
Thereby, simplex $\Delta^{(n)}_{\infty}$ has $k$ vertices if and only if $\prod_{s = n}^{\infty} |z^{(s)}| > 0$. It follows that there exist exactly $k$ finite ergodic invariant measures on the Bratteli diagram $B$ if and only if $\prod_{n = 1}^{\infty} |z^{(n)}| > 0$ or, equivalently, $\sum_{n = 1}^{\infty}(1 - |z^{(n)}|) < \infty$.
\end{proof}

We now consider the case when $B$ is a Bratteli diagram of rank $k \geq 2$ and its incidence matrices $F_n = (f_{v,w}^{(n)})$ have $\rank\ F_n = 2$ for all $n= 1, 2, \ldots$

\begin{thm}
Let $B = (V,E)$ be a Bratteli diagram of rank $k \geq 2$; identify $V_n$ with $\{1, ... ,k\}$ for any $n \geq 1$. Let $F_n = (f_{i,j}^{(n)})$ form a sequence of  incidence matrices of $B$ such that
$\sum_{j \in V_n} f_{i,j}^{(n)} = r_n \geq 2$ for every $i \in V_{n+1}$. Suppose that  $\rank\ F_n = 2$ for all $n$.
Denote
$$
\Delta^{(n)} := \{(x_1, \ldots, x_k)^T : \sum_{i = 1}^k x_i = \frac{1}{r_0 \cdots r_{n-1}}
\ \mbox{ and } \ x_i   \geq 0, \  1 \leq i \leq k \}.
$$

Then there are two invariant measures if and only if
$$
\prod_{n=1}^{\infty} \frac{d(F_n^{T}F_{n+1}^{T}(\Delta^{(n+2)}))}{d(F_{n+1}^{T}(\Delta^{(n+2)}))} > 0
$$
or, equivalently,
$$
\sum_{n=1}^{\infty}\left(1 - \frac{d(F_n^TF_{n+1}^{T}(\Delta^{(n+2)}))}{d(F_{n+1}^{T}(\Delta^{(n+2)}))}\right) < \infty.
$$
The diagram $B = (V,E)$ has a unique invariant measure if and only if
$$
\sum_{n=1}^{\infty}\left(1 - \frac{d(F_n^T F_{n+1}^{T}(\Delta^{(n+2)}))}{d(F_{n+1}^{T}(\Delta^{(n+2)}))}\right) = \infty.
$$
\end{thm}

\begin{proof}
Since rank $F_n = 2$ for every $n$, the set $\Delta^{(n)}_1 = G_{n}(\Delta^{(n+1)})$ is an interval in $\Delta^{(n)}$ and the mappings
$$
\Delta^{(1)}_1 \stackrel{G_1}{\longleftarrow} \Delta^{(2)}_1 \stackrel{G_2}{\longleftarrow} \Delta^{(3)}_1 \stackrel{G_3}{\longleftarrow} \ldots
$$
are one-to-one. We denote by $d_{m+1}^{(n)}$  the length of the interval $\Delta_{m+1}^{(n)} = G_{n} \circ \ldots \circ G_{n+m-1} (\Delta^{(n+m)}_{1})$. Then $d_{m+1}^{(n)} \geq d_{m+2}^{(n)}$ for every $m$ and the intersection $\Delta_{\infty}^{(n)} = \bigcap_{m=0}^\infty \Delta_{m+1}^{(n)}$ is a point if and only if $\lim_{m \rightarrow \infty} d_{m+1}^{(n)} = 0$.  The set  $\Delta_{\infty}^{(n)}$ is an interval if and only if $\lim_{m \rightarrow \infty} d_{m+1}^{(n)} > 0$. In this case we denote  $d(\Delta_{\infty}^{(n)}) = \lim_{m \rightarrow \infty} d_{m+1}^{(n)}$. Then
\begin{eqnarray*}
\frac{d(\Delta_{m+1}^{(n)})}{d(\Delta_{1}^{(n)})} &=& \frac{d(G_n \circ \ldots \circ G_{n+m-1}(\Delta_1^{(n+m)}))}{d(G_n \circ \ldots \circ G_{n+m-2}(\Delta_1^{(n+m-1)}))} \cdot \ldots \cdot \frac{d(G_n(\Delta_1^{(n+1)})}{d(\Delta_1^{(n)})} \\
& = & \frac{d(G_{n+m-1}(\Delta_1^{(n+m)}))}{d(\Delta_1^{(n+m-1)})} \cdot \ldots \cdot \frac{d(G_n(\Delta_1^{(n+1)})}{d(\Delta_1^{(n)})}.
\end{eqnarray*}
Therefore, the diagram $B = (V,E)$ has two invariant measures if and only if
$$
\prod_{n=1}^{\infty} \frac{d(G_n(\Delta_1^{(n+1)}))}{d(\Delta_1^{(n)})} > 0
$$
or, equivalently,
$$
\sum_{n=1}^{\infty}\left(1 - \frac{d(G_n(\Delta_1^{(n+1)}))}{d(\Delta_1^{(n)})}\right) < \infty.
$$
The diagram $B = (V,E)$ has a unique invariant measure if and only if
$$
\sum_{n=1}^{\infty}\left(1 - \frac{d(G_n(\Delta_1^{(n+1)}))}{d(\Delta_1^{(n)})}\right) = \infty.
$$
\end{proof}

We remark that, in the above example, it suffices  to require $\rank\ F_n = 2$ for infinitely many $n$ and then use the procedure of telescoping.

\section{Measure of the path space of a subdiagram}\label{subdiagrams}

Let $B'$ be a subdiagram of a Bratteli diagram $B$. Suppose that a probability measure $\mu$ is defined  on $B$. In this section, we  answer the question when the path space $X_{B'}$ of the subdiagram $B'$ considered as a subset of $X_B$  has positive measure $\mu$. Both cases of vertex and edge subdiagrams will be considered.

Let $B$ be a Bratteli diagram with incidence matrices $\{F_n\}_{n=0}^\infty$, and let $\ov B$ be a vertex subdiagram of $B$ defined by a sequence of subsets $W_n \subset V_n$, the support of  $\ov B$. Let $\mu$ be a probability measure on $X_B$. Denote by $Y_w^{(n)}$ the set of all paths $x = (x_1, \ldots ,x_n, \ldots )$ from $X_B$ which pass through vertex $w \in W_n$ and such that the finite path $(x_1, \ldots ,x_n)$ lies in $\ov B$. We set $Y_{\ov B}^{(n)} = \bigcup_{w \in W_n} Y_w^{(n)}$. Then, obviously, $Y^{(n)}_{\ov B} \supset Y^{(n+1)}_{\ov B}$ for all $n$ and the path space $X_{\ov B}$ of $\ol B$ satisfies the relation:
\begin{equation}\label{path space measure}
 X_{\ov B} =\bigcap_{n=1}^\infty Y^{(n)}_{\ov B}.
\end{equation}

\begin{thm}\label{mu_vertex_sbd}
Let $B$, $\ov B$, $\mu$, $Y^{(n)}_{\ov B}$ be as above. Then the series
$$
S = \sum_{n=1}^{\infty} \sum_{v \in W'_{n+1}} \sum_{w \in W_{n}} f_{v,w}^{(n)} p_v^{(n+1)}\ov h_{w}^{(n)}
$$
is always convergent and  $\mu(X_{\ov B}) = \mu(Y^{(1)}_{\ov B}) - S$. Hence,
$$
(\mu(X_{\overline{B}}) = 0) \   \  \Longleftrightarrow  \  \  (S = \mu(Y^{(1)}_{\ov B})).
$$
\end{thm}

\begin{proof} By definition, we see that $\mu (Y^{(n)}_{\ov B}) = \sum_{w \in W_{n}} \ov h_w^{(n)} p_w^{(n)}$. In what follows we use the equality  $\sum_{w \in W_n} f_{v,w}^{(n)} \ov h_{w}^{(n)} = \ov h_v^{(n+1)}$, where $v \in V_{n+1}$.
\begin{eqnarray*}
\mu (Y^{(n + 1)}_{\ov B}) & = & \sum_{v \in W_{n+1}} \ov h_v^{(n+1)} p_v^{(n+1)} \\
&=& \sum_{v \in W_{n+1}} p_v^{(n+1)} \sum_{w \in W_n} f_{v,w}^{(n)} \ov h_{w}^{(n)}\\
 &=& \sum_{w \in W_n} \ov h_{w}^{(n)} \left( \sum_{v \in W_{n+1}} f_{v,w}^{(n)} p_v^{(n+1)} + \sum_{v \in W'_{n+1}} f_{v,w}^{(n)} p_v^{(n+1)} - \sum_{v \in W'_{n+1}} f_{v,w}^{(n)} p_v^{(n+1)}\right)\\
  &=& \sum_{w \in W_n} \ov h_{w}^{(n)}  \left(p_w^{(n)} - \sum_{v \in W'_{n+1}} f_{v,w}^{(n)} p_v^{(n+1)}\right)  \\
  &=&  \mu (Y^{(n)}_{\ov B}) - \sum_{v \in W'_{n+1}} p_v^{(n+1)} \sum_{w \in W_n} \ov h_{w}^{(n)} f_{v,w}^{(n)}.
\end{eqnarray*}
Because $\mu(X_B) = 1$, we have $\mu(Y^{(1)}_{\ov B}) \leq 1$. It follows from (\ref{path space measure}) that
\begin{eqnarray*}
 \mu(X_{\ov B}) &=& \lim_{n \rightarrow \infty} \mu(Y^{(n)}_{\ov B}) \\
&=& \mu(Y^{(1)}_{\ov B}) + \sum_{n=1}^\infty (\mu(Y^{(n+1)}_{\ov B}) - \mu(Y^{(n)}_{\ov B})).
\end{eqnarray*}
 We finally obtain
\begin{eqnarray*}
\mu(Y^{(1)}_{\ov B}) &=& \mu(X_{\ov B}) + \sum_{n=1}^\infty (\mu(Y^{(n)}_{\ov B}) - \mu(Y^{(n+1)}_{\ov B})) \\
 &=& \mu(X_{\ov B}) + S.
\end{eqnarray*}
 This relation proves that the series $S$ always converges and the theorem holds.
\end{proof}

The following result  is an analogue of Theorem \ref{mu_vertex_sbd} for an edge subdiagram.

\begin{thm}
Let $B$ be a Bratteli diagram with incidence matrices $\{F_n\}_{n=0}^\infty$, and let $\mu$ be a probability measure on $X_B$. For an edge subdiagram   $\ov B \subset B$,  the series
$$
\widetilde{S} = \sum_{n=1}^{\infty} \sum_{v \in V_{n+1}} \sum_{w \in V_{n}} \tl f_{v,w}^{(n)} p_v^{(n+1)}\ov h_{w}^{(n)}
$$
is always convergent and $\mu(X_{\ov B}) = \mu(Y^{(1)}_{\ov B}) - \widetilde{S}$. Hence,
$$
(\mu(X_{\overline{B}}) = 0)  \  \  \Longleftrightarrow  \  \  (\widetilde{S} = \mu(Y^{(1)}_{\ov B})).
$$
\end{thm}

\begin{proof} The proof is similar to that of  Theorem~\ref{mu_vertex_sbd}. Namely, we have
\begin{eqnarray*}
\mu (Y^{(n + 1)}_{\ov B}) & = & \sum_{v \in V_{n+1}} \ov h_v^{(n+1)} p_v^{(n+1)} \\
&=& \sum_{v \in V_{n+1}} p_v^{(n+1)} \sum_{w \in V_n} \overline{f}_{v,w}^{(n)} \ov h_{w}^{(n)}\\
 &=& \sum_{w \in V_n} \ov h_{w}^{(n)} \left( \sum_{v \in V_{n+1}} \overline{f}_{v,w}^{(n)} p_v^{(n+1)} + \sum_{v \in V'_{n+1}} \widetilde{f}_{v,w}^{(n)} p_v^{(n+1)} - \sum_{v \in V'_{n+1}} \widetilde{f}_{v,w}^{(n)} p_v^{(n+1)}\right)\\
  &=& \sum_{w \in V_n} \ov h_{w}^{(n)}  \left(p_w^{(n)} - \sum_{v \in V_{n+1}} \widetilde{f}_{v,w}^{(n)} p_v^{(n+1)}\right)  \\
  &=&  \mu (Y^{(n)}_{\ov B}) - \sum_{v \in V_{n+1}} p_v^{(n+1)} \sum_{w \in V_n} \ov h_{w}^{(n)} \widetilde{f}_{v,w}^{(n)}.
\end{eqnarray*}
Then, we use the same method as in Theorem~\ref{mu_vertex_sbd} to finish the proof.
\end{proof}

\medskip
In the next proposition we consider the case of a stationary Bratteli diagram $B$.
It is worth noting that this proposition is not true in the case of arbitrary diagrams (see Remark \ref{propnottrue} below).

\begin{prop}\label{stationary}
Let $B$ be a stationary Bratteli diagram with  irreducible incidence matrix $F$ and $A = F^T$. Let $\ov B$ be a proper stationary edge subdiagram of $B$. Denote by $\ov F$ the incidence matrix of $\ov B$ and set $\ov A = \ov F^T$.  Let $\mu$ be the unique probability measure on $B$ invariant with respect to the tail equivalence relation $\mathcal E$. Then $\mu(X_{\ov B}) = 0$.
\end{prop}

\begin{proof} It follows from the condition of the proposition that $\ov F < F$ and $\ov A < A$.  Let $\lambda$ be the Perron-Frobenius eigenvalue for $A$ with the corresponding eigenvector $x = (x_v)$. Denote by  $\ov \lambda$  the Perron-Frobenius eigenvalue for $\ol A$. Then $\ov \lambda < \lambda$ (see \cite{G}). If $\ol e$ is a finite path with $r(\ol e) = v \in V_n$, then  $\mu([\ol e]) = \dfrac{x_v}{\lambda^{n-1}}$, where $[\ol e]$ is the cylinder subset of $X_B$ defined by $\ol e$ (see, for instance, \cite{BKMS_2010}). Therefore, we have
\begin{eqnarray*}
 \mu(X_{\ov B})&=& \lim_{n \rightarrow \infty} \mu(Y^{(n)}_{\ov B})\\
   &=& \lim_{n \rightarrow \infty} \sum_{w \in V_n} \mu(Y^{(n)}_w)\\
& = & \lim_{n \rightarrow \infty} \sum_{w \in V_n} \ov h_w^{(n)} \frac{x_w}{\lambda^{n-1}}.
\end{eqnarray*}
We recall that an entry of $\ol A$ is positive if and only if the corresponding entry of $A$ is positive (we do not remove an edge to define $\ol B$ if it is the only one edge between a pair of vertices). Therefore, since $A$ is irreducible, we conclude that $\ov A$ is also irreducible. Then  $\ov h_w^{(n)} \sim \ov \lambda^n$ as $n \rightarrow \infty$ \cite{BKMS_2010}. Hence, there exists some $C>0$ such that $\dfrac{\ov h_w^{(n)}}{\ov \lambda^n} \rightarrow C$ as $n \rightarrow \infty$. Since $\sum_{w \in V_n} \dfrac{\ov \lambda^{n-1}}{\lambda^{n-1}} \ov \lambda x_w \rightarrow 0$ as $n \rightarrow \infty$, we obtain $\mu(X_{\ov B}) = 0$.
\end{proof}

The following theorem gives a necessary and sufficient condition for a  subdiagram $\ol B$ of $B$ to have a path space of zero measure in $X_B$.   Though the theorem is formulated for a vertex subdiagram, the statement remains  true also for any edge subdiagram $\ol B$ (see Remark \ref{edge subdiagram} below).

\begin{thm}\label{main}
Let $B$ be a simple Bratteli diagram, and let $\mu$ be any probability ergodic measure on $X_B$. Suppose that  $\ov B$ is a vertex subdiagram of $B$ defined by a sequence  $(W_n)$ of subsets of $V_n$. Then $\mu(X_{\ov B}) = 0$ if and only if
\begin{equation}\label{thin set condition}
\forall \varepsilon > 0 \ \exists n = n(\varepsilon) \ \mathrm{such\ that} \ \forall w \in W_n \ \mathrm{one\ has} \ \frac{\ov h_w^{(n)}}{h_w^{(n)}} < \varepsilon.
\end{equation}
\end{thm}

 \begin{proof} We first prove the ``if'' part. Fix $\varepsilon > 0$ and find $n = n(\varepsilon)$ such that $\dfrac{\ov h_w^{(n)}}{h_w^{(n)}} < \varepsilon$ for every $w \in W_n$. We note that
$$
\sum_{w \in W_n} h_w^{(n)} p_w^{(n)} < \sum_{w \in V_n} h_w^{(n)} p_w^{(n)} = 1.
$$
Then we have
\begin{eqnarray*}
  \mu(X_{\ov B})  & \leq &\sum_{w \in W_n} \ov h_w^{(n)} p_w^{(n)} \\
   &=& \sum_{w \in W_n} \frac{\ov h_w^{(n)}}{h_w^{(n)}} h_w^{(n)} p_w^{(n)}\\
 &< &\varepsilon.
\end{eqnarray*}
Hence $\mu(X_{\ov B}) = 0$.

To prove the ``only if'' part, we will need the following lemma.

\begin{lemm}\label{lemma}
If $\mu(X_{\ov B}) = 0$, then for every $m > 1$ and every $K > 1$ there exists $N > m$ such that for every $w \in W_m$ and every $v \in W_N$ we have
$$
|E(w,v)| \geq K \ov h_v^{(N)}
$$
(recall that $E(w,v)$ is the set of finite paths between vertices $w$ and $v$).
\end{lemm}

{\em Proof of the lemma}. Since $\mu(X_{\ov B}) = 0$, we get $\sum_{w \in W_n} \ov h_w^{(n)} p_w^{(n)} \rightarrow 0$ as $n \rightarrow \infty$. Take any integers $m$ and  $K$. Then  $\{p_w^{(m)} : w \in W_m\}$ is a finite set of positive numbers. Hence there is $n$ such that $p_w^{(m)} > K \sum_{w \in W_n} \ov h_w^{(n)} p_w^{(n)}$ for every $w \in W_m$. By the ergodic theorem, we have
$$
p_w^{(m)} = \lim_{N \rightarrow \infty} \frac{|E(w,v)|}{h_v^{(N)}},\ \  \mathrm{and}\ \  \
p_u^{(n)} = \lim_{N \rightarrow \infty} \frac{|E(u,v)|}{h_v^{(N)}}
$$
for every $v \in V_N$ and every $u \in W_n$.
Thus, we can find $N$ such that for any $w \in W_m$ and any $v \in V_N$ the following inequality holds:
$$
\frac{|E(w,v)|}{h_v^{(N)}} > K \sum_{u \in W_n} \ov h_u^{(n)} \frac{|E(u,v)|}{h_v^{(N)}}.
$$
Therefore,
\begin{eqnarray*}
|E(w,v)| & > & K \sum_{u \in W_n} \ov h_u^{(n)} |E(u,v)|\\
 & > & K \sum_{u \in W_n} \ov h_u^{(n)} |\ov E(u,v)|\\
& = & K \ov h_v^{(N)},
\end{eqnarray*}
where  we denote by $\ov E(u,v)$  the set of all finite paths in $\ov B$ between the vertices $u$ and $v$. This proves the lemma.
\medskip

We continue now the {\em proof of the theorem}. Suppose that $\mu(X_{\ov B}) = 0$. Take $\varepsilon > 0$ and find $K$ such that $\dfrac{1}{K} < \varepsilon$. By Lemma \ref{lemma}, there is $N$ such that $K \ov h_v^{(N)} < |E(w,v)| < h_v^{(N)}$ for every $w \in W_m$ and for some $m < N$. Thus,
$$
\frac{\ov h_v^{(N)}}{h_v^{(N)}} < \frac{1}{K} < \varepsilon
$$ for every $v \in W_N$. This completes the proof.
\end{proof}

In fact, Theorem \ref{main} states that if  a subdiagram $\ol B$ satisfies  (\ref{thin set condition}), then $X_{\ol B}$ has measure zero with respect to every ergodic invariant measure, that is the set $X_{\ol B}$ is {\em thin} according to the definition from \cite{giordano_putnam_skau:2004}.

\begin{corol}
Let $B$ be a simple Bratteli diagram with  a probability ergodic measure  $\mu$ on $X_B$. Suppose that $\ov B$ is a vertex subdiagram of $B$ defined by a sequence of subsets  $W_n$ of $V_n$. Then

(1) $\mu(X_{\ov B}) > 0$ if and only if there exists $\delta > 0$ such that for all $n > 1$ there is $w_0 \in W_n$ such that $\dfrac{\ov h_{w_0}^{(n)}}{h_{w_0}^{(n)}} > \delta$.

(2) If $\mu(X_{\ov B}) = 0$, then for every $m > 1$ and every $k > 1$ there exists $N > m$ such that for every $w \in W_m$ and every $v \in W_N$ we have
$$
|E(w,v)| \geq K |\ov E(w,v)|.
$$
\end{corol}

\begin{proof}
(1) This is a straightforward corollary of Theorem~\ref{main}.

(2) Suppose that $\mu(X_{\ov B}) = 0$. Then, by Lemma~\ref{lemma}, for every $m > 1$ and every $k > 1$ there exists $N > m$ such that for every $w \in W_m$ and every $v \in W_N$ we have $|E(w,v)| \geq K \ov h_v^{(N)}$. It is obvious that $ \ov h_v^{(N)} > |\ov E(w,v)|$. Hence $|E(w,v)| \geq K |\ov E(w,v)|$.
\end{proof}

Theorem \ref{main} states, in other words, that the structural properties of the diagram $B$ determine whether the measure of a path space of a subdiagram $\ol B$ is  zero. Another corollary of this result shows that the extension of {\em any} measure $\ol \mu$ from $X_{\ol B}$ must be infinite.

\begin{corol}\label{corol_infty}
Let $\ol B$ be a subdiagram of $B$ such that $X_{\ol B}$ is a thin subset of $X_B$.  Then for any probability invariant measure $\ov \mu$ on $\ov B$ we have $\widehat{\overline{\mu}}(\wh X_{\ov B}) = \infty$.
\end{corol}

\begin{proof} Assume that the converse holds, i.e., there exists $M$ such that $\widehat{\overline{\mu}}(\wh X_{\ov B}) < M$. Take $\varepsilon > 0$ such that $\dfrac{1}{\varepsilon} > M$.  Given $\varepsilon > 0$,  we can find $n = n(\varepsilon)$, by Theorem \ref{main}, such that $\dfrac{\ov h_w^{(n)}}{h_w^{(n)}} < \varepsilon$ for every $w \in W_n$. Then
\begin{eqnarray*}
\widehat{\overline{\mu}}(\wh X_{\ov B}) & > &\sum_{w \in W_n} h_w^{(n)} \ov p_w^{(n)}\\
 & = &\sum_{w \in W_n} \frac{h_w^{(n)}}{\ov h_w^{(n)}} \ov h_w^{(n)} \ov p_w^{(n)}\\
& > & \frac{1}{\varepsilon} \sum_{w \in W_n} \ov h_w^{(n)} \ov p_w^{(n)}\\
& > & M.
\end{eqnarray*}
This is a contradiction.
\end{proof}

\begin{remar}\label{edge subdiagram} It is not hard to see that
Theorem~\ref{main} and Corollary~\ref{corol_infty} hold also in the case when $\ov B$ is an edge subdiagram. The proofs are analogous to the case of vertex subdiagrams.
\end{remar}

\begin{remar}\label{propnottrue}
(1) We first note that Proposition \ref{stationary} holds also for vertex subdiagrams. Indeed, if $\ol B$ is a vertex subdiagram of a stationary simple Bratteli diagram, then $\mu (X_{\ol B}) =0$, where $\mu$ is a unique probability invariant measure on $B$.

(2) On the other hand, Proposition~\ref{stationary} does not hold in case of Bratteli  diagrams $B$ of finite rank. More precisely, there are vertex subdiagrams $\ol B$ of finite rank Bratteli diagrams whose path spaces $X_{\ol B}$ are of positive measure in $X_B$ (see details in \cite{BKMS_2013}).

In order to illustrate this fact, we recall  Example~\ref{ERS,ECS}. One can easily find an example of a Bratteli diagram $B$ and its subdiagram $\ov B$ such that $\wh{\ov \mu} (\wh X_{\ov B})$ is finite. For instance, let
$$
F_n =
\begin{pmatrix}
1 & \ldots & 1 & 2\\
\vdots & \ddots & \vdots & \vdots\\
1 & \ldots & 1 & 2
\end{pmatrix}
$$
 be a matrix with the ERS property. Take
$$
\ov F_n =
\begin{pmatrix}
1 & \ldots & 1 & 1\\
\vdots & \ddots & \vdots & \vdots\\
1 & \ldots & 1 & 1
\end{pmatrix}.
$$
Then $\ov F_n$ has the ECS property and, by relation (\ref{ERS,ECS,neccsuff}),
$$
\left(\widehat{\overline{\mu}}(\widehat{X}_{\overline{B}}) < \infty \right) \ \Longleftrightarrow \
\left(\sum_{n=1}^\infty \frac{r_0 \ldots r_{n-1}}{\overline{c}_0 \ldots \overline{c}_{n}} |V_{n+1}| = \sum_{n=1}^\infty \prod_{i=0}^{n-1} \frac{|V_i|+1}{|V_{i+1}|}  < \infty\right).
$$
If, for instance, $|V_i| = 2^i$, then we obtain $\widehat{\overline{\mu}}(\widehat{X}_{\overline{B}}) < \infty$. Note that in the case of stationary diagrams $B$ and $\ov B$, the measure $\wh{\ov \mu} (\wh X_{\ov B})$ is always infinite.
\end{remar}


The procedure of measure extension that was regularly used above can be interpreted in the following way. Let $\ol B$ be a Bratteli subdiagram with path space $X_{\ol B}$ and an ergodic probability measure $\ol \mu$ on it such that the measure extension $\wh{\ol\mu}$ is finite. We have a sequence of clopen sets $Y^{(n)}_{\ov B}$ such that $\bigcap_n Y^{(n)}_{\ov B} =X_{\ol B}$. When we extend $\ol \mu$ to $\wh{\ol\mu}$ we work consequently with cylinder sets taken from the sets $Y^{(n)}_{\ov B}$ and construct the measure extension. On the other hand, we could use the measure space $(X_{\ol B}, \ol\mu)$ and a Vershik map $T$ acting on $X_B$ to define a partition of $X_B$ into the towers constructed by the first return function; this construction is a classical one in the ergodic theory. We remark that in order to use the notion of a Vershik map, one needs to  turn $B$ into an ordered Bratteli diagram (see, for instance, \cite{BKY14} for more information on orders on Bratteli diagrams). Using the first return function we simultaneously define an extension of  $\ol\mu$  by $T$-invariance to an ergodic measure $\nu$ on  $X_B$.  Our result (see below) states that these two constructions give the same measure on $X_B$.

\begin{prop}
Let $\ov B = \ov B(W_n)$ be a vertex subdiagram of a Bratteli diagram $B = (V, E)$, and let  $\ov \mu$ be a finite ergodic measure on the path space of $\ov B$ such that $\widehat{\overline{\mu}}(\wh X_{\ov B}) < \infty$.
Assume that $|E(v,w)| \geq 1$ for every $v \in V_{n+1}$ and $w \in V_n$ ($n \in \mathbb{N}$).
Let $\omega$ be any order on $B$ such that the sets  of maximal and minimal
paths, $X_{\max}(\omega)$ and $X_{\min}(\omega)$, have $\wh{\ol\mu}$-measure zero\footnote{The definition of notions used here  can be found in \cite{BKY14}.}.
Let $\nu$ the measure on $X_B$ obtained
from $\ov \mu$ using the first return function for the Vershik map $T_{\omega}$.
Then the measures $\widehat{\overline{\mu}}$ and $\nu$ are equivalent.
\end{prop}

\begin{proof}

We first observe that, for a given diagram $B$, there always exists an order
$\omega$ for which the condition of the theorem holds.

Given such an $\omega$, the Vershik map $T = T_\omega \colon X_B \setminus X_{\max} \rightarrow X_B \setminus X_{\min}$ is well defined.
Then $T$ defines the first return function $n(x)$ on $X_{\ov B} \setminus X_{\max}$ as follows: $n(x) = \min\{n \geq 1 \colon T^n(x) \in X_{\ov B} \setminus X_{\min}\}$ for $x \in X_{\ov B} \setminus X_{\max}$.

Let $e = (e_1, \ldots, e_n)$ be a finite path in $\ov B$ which ends in a vertex $w \in W_n$. Since the set of all such paths is ordered, we can consider all successors of
$e$ which also belong to $\ol B$.  Suppose that this set is nonempty and
denote by $e'$ the nearest successor from this set.  Denote by $X_{w}^{(n)}(e)$ a cylinder set in $X_B$ corresponding to $e$. There exists an integer $k > 0$ such that
$X_{w}^{(n)}(e') = T^k(X_{w}^{(n)}(e))$. Denote $\ov X_{w}^{(n)}(e) = X_{w}^{(n)}(e) \cap X_{\ov B}$.
Then for a path $x \in \ov X_{w}^{(n)}(e)$ we set $n(x) = k$.
Let $E_w^{(n)} \in \ov X_{w}^{(n)}$ be the cylinder set that does not have a successor in $\ov X_{w}^{(n)}$. Thus, $E_w^{(n)}$ is generated by the path with the largest assigned number amongst all finite paths $(x_1, \ldots ,x_n)$ that lie in $\ov B$ with $r(x_n) = w \in W_n$.
We have defined the value of the return function $n(x)$ on each set $\ov X_{w}^{(n)} \setminus E_w^{(n)}$, $w \in W_n$ and $n \in \mathbb{N}$. Let
$$
I_n = \bigcup_{w \in W_n} \ov X_{w}^{(n)} \setminus E_w^{(n)}.
$$
Then $I_n \subset I_{n+1}$, and this means that the return function is well-defined. Indeed, the first return function $n(x)$ is defined on the set
$$
\bigcup_{n=1}^{\infty} \bigcup_{w \in W_n} (\ov X_w^{(n)} \setminus E_{w}^{(n)}) = \bigcup_{n=1}^{\infty} I_n.
$$
Let $\ov h^{(n)}_{\min} = \min\{\ov h_w^{(n)}, w \in W_n\}$. Since $|E(v,w)| \geq 1$ for every $v \in V_{n+1}$, $w \in V_n$ and $n \in \mathbb{N}$, we have $\ov h^{(n)}_{\min} \rightarrow \infty$ as $n \rightarrow \infty$. Since $\ov \mu(X_{\ov B}) = \sum_{w \in W_n}\ov h_w^{(n)}\ov p_w^{(n)} = 1$, we have
$$
\ov \mu( \bigcup_{w \in W_n} E_{w}^{(n)}) = \sum_{w \in W_n} \ov p_w^{(n)} \leq \frac{1}{\ov h_{min}^{(n)}} \rightarrow 0
$$
as $n \rightarrow \infty$. This implies that $\ov \mu(I_n) \rightarrow 1$ and the return function is defined on $X_{\ov B}$ except for a subset of $\ov \mu$-measure zero.

Let the cylinder set $E_w^{(n)}$ be generated by a finite path $(x_1, \ldots, x_n)$ that lies in $\ov B$. Recall that this path does not have successor in $\ov B$.
Let $l_w^{(n)} - 1$ be the amount of paths $(y_1, \ldots, y_n)$ in $B$ such that $r(y_n) = w \in W_n$ and each of these paths $(y_1, \ldots ,y_n)$ has the assigned order number greater than that of $(x_1, \ldots, x_n)$. Thus, these $l_w^{(n)} - 1$ paths do not belong to $\ov B$.
For $w \in W_n$, let $\wh Y_w^{(n)}$ be the set $\wh X_w^{(n)}$ without last $l_w^{(n)}$ cylinder sets, indicated above.
For every $n \geq 1$, the function $n(x)$ determines a $T$-tower $\wh Y^{(n)}$ over the set $\bigcup_{w \in W_n} (\ov X_w^{(n)} \setminus E_{w}^{(n)})$. Then for every $w \in W_n$, the set $\wh Y^{(n)}$ contains all cylinder sets of the tower $\wh X_w^{(n)}$ except for the last $l_w^{(n)}$ sets. Let $\nu_n$ be the measure extended by invariance  from the measure $\ov \mu$ on $\bigcup_{w \in W_n} (\ov X_w^{(n)} \setminus E_{w}^{(n)})$. We have
$$
\nu_n (\wh Y^{(n)}) = \sum_{w \in W_n} \ov p_w^{(n)}(h_w^{(n)} - l_w^{(n)}).
$$

It is easy to see that $\wh Y^{(n)} \subset \wh Y^{(n+1)}$. Denote $\wh Y = \bigcup_{n=1}^{\infty}\wh Y^{(n)}$. Then $\wh Y$ is a skyscraper over $X_{\ov B} \setminus \bigcap_{n=1}^{\infty}(\bigcup_{w \in W_n} E_w)$, i.e. $\wh Y = \mathcal{E}(X_{\ov B} \setminus \bigcap_{n=1}^{\infty}(\bigcup_{w \in W_n} E_w))$. Moreover, $\wh Y \subset \wh X_{\ov B}$ and the measure $\nu$ on $\wh Y$, where $\nu = \lim_{n \rightarrow \infty} \nu_n$, coincides with the measure $\wh {\ov \mu}$.
To prove that the extension construction of the measure $\ov \mu$ coincides with the above ``classical'' skyscraper construction, it suffices to show that $\wh {\ov \mu} (\wh X_{\ov B} \setminus \wh Y) = 0$.

Take $v \in W_{n+1}$. Then the tower $\wh X_v^{(n+1)}$ consists of some subcolumns of the $\tl T$-towers $\wh X_{w_1}^{(n)}, \ldots, \wh X_{w_k}^{(n)}$. Denote these subcolumns by $Z_{w_1}^{(n)}, \ldots, Z_{w_k}^{(n)}$. Let $w_{n,v} \in W_n$ be the ``last vertex'' among $w_1, \ldots, w_k$ belonging to $W_n$, i.e. if $w_j = w_{n,v} $ then $w_{j+1}, \ldots, w_k \in V_n \setminus W_n$. Here the order on vertices is induced by the order on edges that end in $v$.
We have
\begin{equation}\label{ineqlwv}
l_v^{(n+1)} \leq \sum_{w \in W_n'}f_{vw}^{(n)}h_w^{(n)} + l_{w_{n,v}}^{(n)}.
\end{equation}
Applying~(\ref{ineqlwv}) to $l_{w_{n,v}}^{(n)}$ and so on, we find  uniquely a  sequence of verices $\{w_{l,v}\}_{l = 1}^n$ such that
$$
l_v^{(n+1)} \leq \sum_{u_n \in W_n'}f_{vu_n}^{(n)}h_{u_n}^{(n)} + \sum_{u_{n-1} \in W_{n-1}'}f_{w_{n,v}u_{n-1}}^{(n-1)}h_{u_{n-1}}^{(n-1)} + \ldots + \sum_{u_1 \in W_1'}f_{w_{2,v}u_1}^{(1)}h_{u_1}^{(1)} + f^{(0)}_{w_{1,v}v_0}.
$$
Denote $M_l = \max_{w \in W_{l+1}}\sum_{u \in W_l'}f_{wu}^{(l)}h_{u}^{(l)}$. Then $M_l = \sum_{u \in W_l'}f_{w_{l+1}u}^{(l)}h_{u}^{(l)}$ for some $w_{l+1} \in W_{l+1}$. Thus, we can write
$$
l_v^{(n+1)} \leq \sum_{u_n \in W_n'}f_{vu_n}^{(n)}h_{u_n}^{(n)} + \sum_{u_{n-1}
\in W_{n-1}'}f_{w_{n}u_{n-1}}^{(n-1)}h_{u_{n-1}}^{(n-1)} + \ldots + \sum_{u_1
\in W_1'}f_{w_{2}u_1}^{(1)}h_{u_1}^{(1)} + f^{(0)}_{w_{1}v_0},
$$
where, in the right hand part of the above relation, all summands but the first one are independent of $v$. Thus, we obtain
\begin{multline*}
\sum_{v \in W_{n+1}} l_v^{(n+1)} \ov p^{(n+1)}_v \leq \sum_{v \in W_{n+1}, u \in W_n'} f_{vu_n}^{(n)}h_{u_n}^{(n)}\ov p^{(n+1)}_v + \\ \sum_{v \in W_{n+1}}\ov p^{(n+1)}_v \left( \sum_{u_{n-1} \in W_{n-1}'}f_{w_{n}u_{n-1}}^{(n-1)}h_{u_{n-1}}^{(n-1)} + \ldots + \sum_{u_1 \in W_1'}f_{w_{2}u_1}^{(1)}h_{u_1}^{(1)} + f^{(0)}_{w_{1}v_0} \right).
\end{multline*}
Rewrite the latter inequality in the form
\begin{multline}\label{long}
\sum_{v \in W_{n+1}} l_v^{(n+1)} \ov p^{(n+1)}_v \leq \sum_{v \in W_{n+1}, u \in W_n'} f_{vu_n}^{(n)}h_{u_n}^{(n)}\ov p^{(n+1)}_v + \\ \sum_{v \in W_{n+1}}\ov p^{(n+1)}_v \left( \frac{1}{\ov p_{w_n}^{(n)}}\sum_{u_{n-1} \in W_{n-1}'}f_{w_{n}u_{n-1}}^{(n-1)}h_{u_{n-1}}^{(n-1)}\ov p_{w_n}^{(n)}  + \ldots + \right.\\
\left. \frac{1}{\ov p_{w_2}^{(2)}}\sum_{u_1 \in W_1'}f_{w_{2}u_1}^{(1)}h_{u_1}^{(1)}\ov p_{w_2}^{(2)} + f^{(0)}_{w_{1}v_0} \right).
\end{multline}
Recall that $|E(v,u)| \geq 1$ for every $v \in V_{n+1}$ and $u \in V_n$. Hence 
$$
\ov p_u^{(n)} = \sum_{v \in V_{n+1}} f_{vu}^{(n)} \ov p_v^{(n+1)} \geq \sum_{v \in V_{n+1}} \ov p_v^{(n+1)}
$$
and $\ov p_u^{(n)} \leq \ov p_w^{(l)}$ for any $w \in W_l$ and any $l \leq n$. Thus, we get
\begin{equation}\label{one}
\frac{\sum_{v \in V_{n+1}} \ov p_v^{(n+1)}}{\ov p_w^{(l)}} \leq 1
\end{equation}
for $l = 1, 2, \ldots, n$. Let
$$
K = \sum_{n = 1}^{\infty} \sum_{v \in W_{n+1}}\sum_{w \in W_n'} f_{vw}^{(n)}h_w^{(n)}\ov p_v^{(n+1)}.
$$
Since $\widehat{\overline{\mu}}(\wh X_{\ov B}) < \infty$, we have $K < \infty$ by Theorem~\ref{thm from BKK}. Therefore, given $\varepsilon > 0$ we can find $l_0$ such that
\begin{equation}\label{firsthalf}
\sum_{l = l_0}^{n-1}\sum_{u \in W_l'}f_{w_{l+1}u}^{(l)}h_u^{(l)}\ov p_{w_{l+1}}^{(l+1)} + \sum_{v \in W_{n+1}}\sum_{u \in W_n'}f^{(n)}_{vu}h_u^{(n)}\ov p_v^{(n+1)} < \frac{\varepsilon}{2}
\end{equation}
for any $n \geq l_0 + 1$. Since $\sum_{v \in W_n} \ov p^{(n)}_v$ tends to zero as $n$ tends to infinity, we can choose $n_0 \geq l_0 + 1$ such that
\begin{equation}\label{epsilonover2k}
\frac{1}{\ov p_w^{(l)}} \sum_{v \in W_{n+1}} \ov p^{(n+1)}_v < \frac{\varepsilon}{2K}
\end{equation}
for every $n \geq n_0$ and $l = 1, \ldots, l_0$. From inequality~(\ref{long}), using~(\ref{firsthalf}), (\ref{one}) and (\ref{epsilonover2k}), we get
\begin{multline*}
\sum_{v \in W_{n+1}} l_v^{(n+1)} \ov p^{(n+1)}_v \leq \left( \sum_{v \in W_{n+1}, u \in W_n'} f_{vu_n}^{(n)}h_{u_n}^{(n)}\ov p^{(n+1)}_v + \sum_{l = l_0+1}^{n-1} \sum_{u_{l} \in W_{l}'}f_{w_{l+1}u_{l}}^{(l)}h_{u_{l}}^{(l)}\ov p_{w_{l+1}}^{(l+1)}\right) \\
+ \frac{\varepsilon}{2K} \left(\sum_{l=1}^{l_0} \sum_{u_l \in W_l'}f_{w_{l+1}u_l}^{(1)}h_{u_l}^{(l)}\ov p_{w_{l+1}}^{(l+1)} + f^{(0)}_{w_{1}v_0} \right) < \varepsilon
\end{multline*}
for $n \geq n_0$. Hence,
$$
\sum_{w \in W_n} l_w^{(n)}\ov p_w^{(n)} \rightarrow 0 \mbox { as } n \rightarrow \infty.
$$
\begin{eqnarray*}
\wh\mu(\wh X_{\ov B} \setminus \wh Y) &=& \wh\mu(\wh X_{\ov B}) -  \wh\mu(\wh Y)\\
&=& \lim_{n \rightarrow \infty} \sum_{w \in W_n} \ov p_w^{(n)}h_w^{(n)} - \lim_{n \rightarrow \infty} \sum_{w \in W_n} \ov p_w^{(n)}(h_w^{(n)} - l_w^{(n)}) \\
&=& \lim_{n \rightarrow \infty} \sum_{w \in W_n} \ov p_w^{(n)}l_w^{(n)}\\
&=& 0.
\end{eqnarray*}
Therefore, $ \wh\mu(\wh X_{\ov B}) =  \wh\mu(\wh Y) = 1$.

\end{proof}

\section{Example}\label{Example section}

In this section, we deal with a class of Bratteli diagrams for which our main results can be applied.

Let $B$ be a Bratteli diagram defined by the sequence of incidence matrices
$$
F_n =
\begin{pmatrix}
a_n& 1 & \ldots & 1 \\
0 & 1 & \ldots & 1 \\
\vdots& \vdots & \ddots & \vdots\\
0& 1 & \ldots & 1\\
1&1 & \ldots & 1
\end{pmatrix}.
$$
We assume that they have the ECS property, that is $a_n + 1 = |V_{n+1}|$ for all $n\geq 0$. Suppose for simplicity that $a_0 = 1$. There are two natural vertex subdiagrams of $B$: one of them, $\ov B_1$, is supported by the first vertex of each level and represents the odometer $(a_n)$, the other one, $\ov B_2$, contains all vertices except the first vertex and is similar to that that was studied  in Example \ref{ex1}. We will consider the measure $\mu$ on $B$ defined by its values $p_w^{(n)} = \dfrac{1}{|V_1| \cdots |V_n|}$ on cylinder sets.

For $\ov B_1$, we have $\ov F_n = (a_n)$. Then
\begin{prop}\label{meas_of_sbd}
$\mu(X_{\ov B_1}) = 0$ if and only if $\sum_{i = 0}^\infty \dfrac{1}{a_i} = \infty$.
\end{prop}
\begin{proof} Indeed, we have
$$
\mu(X_{\ov B_1}) = \lim_{n \rightarrow \infty} \mu(Y_1^{(n)}) = \lim_{n \rightarrow \infty} \ov h_1^{(n)} p_1^{(n)} = \lim_{n \rightarrow \infty} \frac{a_0 \cdots a_{n-1}}{|V_1| \cdots |V_n|} = \prod_{i= 0}^\infty \frac{a_i}{(a_i + 1)}.
$$
Hence, $\mu(X_{\ov B_1}) = 0$ if $\sum_{i = 0}^\infty \dfrac{1}{a_i} = \infty$ and $\mu(X_{\ov B_1}) \in (0,1)$ if $\sum_{i = 0}^\infty \dfrac{1}{a_i} < \infty$.
\end{proof}

\begin{remar}
In fact, one can state even more. In notation used in Theorem \ref{mu_vertex_sbd}, we can conclude that the relation
\begin{eqnarray*}
  \mu(Y_1^{(1)}) &=& \mu(X_{\ov B_1}) + \sum_{n=1}^{\infty} \sum_{v \in W'_{n+1}} \sum_{w \in W_{n}} f_{v,w}^{(n)} p_v^{n+1}\ov h_{w}^{(n)}  \\
  &=& \mu(X_{\ov B_1}) + \sum_{n=1}^{\infty} \frac{a_0 \cdots a_{n-1}}{(1+a_0) \cdots (1 + a_n)}
\end{eqnarray*}
holds. It follows that the series $S = \sum_{n=1}^{\infty} \dfrac{a_0 \ldots a_{n-1}}{(1+a_0) \cdots (1 + a_n)}$ converges for any integers $a_n > 1$; and if $\mu(X_{\ov B_1}) = 0$, then the sum $S = \mu(Y_1^{(1)}) = \dfrac{a_0}{1 + a_0}$ depends only on $a_0$.
\end{remar}

\begin{prop}\label{meas_ext12}
If $\mu(X_{\ov B_1}) = 0$, then $\widehat{\overline{\mu}}(\wh X_{\ov B_1}) = \infty$.
\end{prop}

\begin{proof}
We identify the vertices of $V_n$ with the set $\{1, ... , |V_n|\}$ for every $n$.
Since $\mu(X_{\ol B_1}) = 0$, we have that, by Proposition \ref{meas_of_sbd},
$\sum_{n\geq 0}a_n^{-1} =\infty$.

Let the matrix $G_n < F_n$ be such that $g^{(n)}_{v,w} = f^{(n)}_{v,w}$ for all vertices $v, w$ but $g^{(n)}_{|V_{n+1}|,1} =0$  for every $n$. If $k_w^{(n)}$ denotes the vector of heights corresponding to the subdiagram with incidence matrices $G_n$, then
\begin{eqnarray*}
S &= &\sum_{n=1}^\infty \frac{1}{a_0 \cdots a_n}\sum_{w \neq 1} h_w^{(n)}\\
 &> &\sum_{n=1}^\infty \frac{1}{a_0 \cdots a_n} k_w^{(n)}(|V_n| - 1) \\
&= & \sum_{n=1}^\infty \frac{|V_1| \cdots |V_{n-1}|}{a_0 \cdots a_n} a_{n-1}\\
 &= &\sum_{n=1}^\infty \frac{(1 + a_0) \cdots (1 + a_{n-2})}{a_0 \cdots a_{n-2}} \frac{1}{a_n}\\
  &> &\sum_{n=1}^\infty \frac{1}{a_n}.
\end{eqnarray*}
 Hence $S > \sum_{n = 0}^\infty \dfrac{1}{a_n} = \infty$ and $\widehat{\overline{\mu}}(\wh X_{\ov B_1}) = \infty$.
\end{proof}

For a subdiagram $\ov B = \ov B_2$ we can prove a statement that is analogous to Proposition \ref{meas_of_sbd}

\begin{prop}\label{meas_of_sbd2}
$\mu(X_{\ov B_2}) = 0$ if and only if $\sum_{i = 0}^\infty \dfrac{1}{a_i} = \infty$.
\end{prop}

\begin{proof}
We have
\begin{eqnarray*}
\mu(X_{\ov B_2}) &= &\lim_{n \rightarrow \infty} \sum_{w \neq 1} \ov h_w^{(n)}p_w^{(n)}\\
&= &\lim_{n \rightarrow \infty} \sum_{w \neq 1} \frac{|W_1| \cdots |W_{n-1}|}{|V_1| \cdots |V_n|}\\
&= & \lim_{n \rightarrow \infty} \prod_{i = 1}^n\left( \frac{|W_i|}{|V_i|}\right)\\
 &= & \lim_{n \rightarrow \infty} \prod_{i = 1}^n\left(1 - \frac{1}{|V_i|}\right)\\
   &= &\prod_{i = 0}^\infty \left( 1 - \frac{1}{1 + a_i}\right).
\end{eqnarray*}
Hence $\mu(X_{\ov B_2}) = 0$ if $\sum_{i = 0}^\infty \dfrac{1}{a_i} = \infty$, and $\mu(X_{\ov B_2}) \in (0,1)$ if $\sum_{i = 0}^\infty \dfrac{1}{a_i} < \infty$.
\end{proof}

The idea of the proof of next result is similar to Proposition \ref{meas_ext12}, so that we omit its proof.
\begin{prop}
If $\mu(X_{\ov B_2}) = 0$, then $\widehat{\overline{\mu}}(\wh X_{\ov B_2}) = \infty$.
\end{prop}

To finish our study of the diagram $B$, we show how one can  find all ergodic measures on $B$.
We recall that any finite ergodic invariant measure $\mu$ on $X_B$ is determined by the sequence $(p^{(n)})$ of  values of measure $\mu$ on cylinder sets at each level $n$. That is,
$$
p_w^{(n)} = \sum_{v \in V_{n+1}}f_{v,w}^{(n)}p_v^{(n+1)}.
$$
It is easy to see that in our example the vector  $p^{(n)}$ has the form $(p_1^{(n)}, p_0^{(n)}, ... ,p_0^{(n)})$. Hence, for $n \geq 1$, we obtain the system of equations
$$
\left\{
\begin{aligned}
p_1^{(n)} = a_n p_1^{(n+1)} + p_0^{(n+1)},\\
p_0^{(n)} = p_1^{(n+1)} + a_n p_0^{(n+1)}.\\
\end{aligned}
\right.
$$
For $n = 1$, we have $a_0 p_1^{(1)} + a_0 p_0^{(1)} = 1$, hence $p_1^{(1)} + p_0^{(1)} = \dfrac{1}{a_0}$. Taking the sum of these equations, we obtain $p_1^{(n)} + p_0^{(n)} = (a_n + 1)(p_1^{(n+1)} + p_0^{(n+1)})$. Therefore,
 $$
 p_1^{(n)} + p_0^{(n)} = \frac{1}{a_0(1 + a_1)\cdots (1+a_{n-1})}.
 $$
  Set $c_n = p_1^{(n)} + p_0^{(n)}$. Define a linear transformation $T_n$ acting on $\mathbb{R}^2$ as follows: $T_n \colon (x,y) \mapsto (a_n x + y, x + a_n y)$. We get the vectors $T_n(c_{n+1},0) = P^{(1)}_n$ and $T_n(0, c_{n+1}) = Q^{(1)}_n$. Similarly, we  denote  $P^{(m)}_n = T_n \circ \ldots \circ T_{n+m-1}(c_{n+m},0)$ and $Q^{(m)}_n = T_n \circ \ldots \circ T_{n+m-1}(0, c_{n+m})$. Let $I_n^{(m)}$ be the interval on the plane with endpoints $P^{(m)}_n, Q^{(m)}_n$. Then it is obvious that  $I_n^{(1)} \supset I_n^{(2)} \supset \ldots \supset I_n^{(m)} \supset \ldots $; and the ergodic measures on $B$ correspond to the endpoints of the interval $I_n^{(\infty)} = \lim_{m \rightarrow \infty} I_n^{(m)}$. In such a way we see that there exist at most two ergodic measures on $B$. We prove the following result.

\begin{thm}
If the extensions of the ergodic measures supported by subdiagrams $\ov B_i \ (i=1,2)$ are both finite, then there are exactly two finite ergodic measures on $B$, and they coincide with these extensions. Otherwise, there is a unique finite ergodic measure defined by the ECS property.
\end{thm}

\begin{proof}
Let $I_n$ be the interval in $\mathbb{R}^2$ with endpoints $(0,c_n)$ and $(c_n,0)$. Denote by $|I_n|$ the length of $I_n$. We will need the following lemma

\begin{lemm}
If $\sum_{n=1}^{\infty} \dfrac{1}{a_n} = \infty$, then there is a unique ergodic invariant measure on $B$. If $\sum_{n=1}^{\infty} \dfrac{1}{a_n} < \infty$, then there are exactly two different ergodic invariant measures on $B$.
\end{lemm}

{\em Proof of the lemma}. We have
$$
\frac{|I_n^{(1)}|}{|I_n|} = \frac{\sqrt{2}(a_n - 1)c_{n+1}}{\sqrt{2} c_n} = \frac{a_n - 1}{a_n + 1} = 1 - \frac{2}{1 + a_n}.
$$
Hence
$$
\frac{|I_n^{(m)}|}{|I_n|} = \frac{|I_n^{(m)}|}{|I_n^{(m-1)}|}\ldots \frac{|I_n^{(0)}|}{|I_n|} = \prod_{k=n}^m\left(1 - \frac{2}{1+a_k}\right).
$$
Thus,
$$
\frac{|I_n^{(\infty)}|}{|I_n|} =  \prod_{k=n}^\infty\left(1 - \frac{2}{1+a_k}\right).
$$
Therefore, $\dfrac{|I_n^{(\infty)}|}{|I_n|} \rightarrow 0$ if and only if $\sum_{n=1}^{\infty} \dfrac{1}{a_n} = \infty$. In this case $I_n^{(\infty)}$ is just a point, hence there is a unique ergodic invariant measure. Otherwise, there are two ergodic measures. The lemma is proved.

\medskip
We continue the {\em proof of the theorem}. We show that in the case of two ergodic measures, they coincide with the extensions of ergodic measures from subdiagrams $\ov B_i$. Set $P_n = (a_n c_{n+1}, c_{n+1})$ and $Q_n = (c_{n+1}, a_n c_{n+1})$. Let $P^{(n)}, Q^{(n)}$ be the endpoints of $I_n^{(\infty)}$.
Then we have $|P^{(n)}Q^{(n)}| = \prod_{k \geq n}\left(1 - \dfrac{2}{1 + a_k}\right)\sqrt{2}c_n$. The coordinates of $P^{(n)}$ are

\begin{eqnarray*}
P^{(n)} &= &\left(\frac{c_n}{2} + \frac{|P^{(n)}Q^{(n)}|}{2\sqrt{2}}, \frac{c_n}{2} - \frac{|P^{(n)}Q^{(n)}|}{2\sqrt{2}}\right)\\
&= & \left(\frac{c_n}{2} \left(1 + \prod_{k=n}^\infty\left(1 - \frac{2}{1+a_k} \right)\right), \frac{c_n}{2} \left(1 + \prod_{k=n}^\infty\left(1 - \frac{2}{1+a_k} \right)\right)\right).
\end{eqnarray*}
Let $\mu_P$ be the ergodic measure corresponding to $P$.
Recall that $\mu_P$ is defined by its values $(p_1^{(n)},p_0^{(n)},\ldots, p_0^{(n)})$, where $p_1^{(n)}$ is the $x$-th coordinate of $P^{(n)}$ and $p_0^{(n)}$ is the $y$-th coordinate. Then
\begin{eqnarray*}
\mu_P(X_{\ov B_1}) &= &\lim_{n \rightarrow \infty} p_1^{(n)}\ov h_1^{(n)} \\
&= & \lim_{n \rightarrow \infty} \frac{c_n}{2} \left(1 + \prod_{k=n}^\infty\left(1 - \frac{2}{1+a_k} \right)\right) a_0 \cdots a_n\\
&= & \frac{1}{2} \lim_{n \rightarrow \infty} \frac{a_0 \cdots a_{n-1}}{a_0(1+a_1)\cdots (1+a_{n-1})} \left(1 + \prod_{k=n}^\infty\left(1 - \frac{2}{1+a_k} \right)\right)\\
&> & 0.
\end{eqnarray*}
Recall that $\ov B_1$ is the odometer with edges $(a_n)$. Hence, we have two non-zero finite invariant measures on the odometer $\ov B_1$, namely, $\ov \mu_{B_1}$ and $\mu_P$. Since odometer is a uniquely ergodic system, the measures are equivalent. Thus, $\mu_P = C_1 \wh \mu_{\ov B_1}$, where $C_1$ is a constant multiple. In the same way, one can prove that the other ergodic measure $\mu_Q = C_2 \wh \mu_{\ov B_2}$.
\end{proof}

\textbf{Acknowledgment.} The authors are grateful to the Nicolas Copernicus University, the Max Planck Institute for Mathematics, and the University of Iowa for the hospitality and support. We would like also to thank Tomasz Downarowicz, Palle Jorgensen, Constantyne Medynets,  and Mykola Matviichuk for useful discussions.

\end{document}